\documentclass[reqno,11pt]{amsart}
\usepackage[dvips]{color}
\usepackage{amsmath, latexsym, amsfonts, amssymb, amsthm, amscd,mathrsfs}
\usepackage{graphics,epsf,psfrag}
\numberwithin{equation}{section}

\newtheorem{atheorem}{Theorem}

\newtheorem{theorem}{Theorem}[section]
\newtheorem{fact}[theorem]{Fact}
\newtheorem{lemma}[theorem]{Lemma}
\newtheorem{proposition}[theorem]{Proposition}
\newtheorem{corollary}[theorem]{Corollary}
\newtheorem{remark}[theorem]{Remark}
\newtheorem{definition}[theorem]{Definition}

\newtheorem{example}{Example}

\newcommand{\ind}{\mathbf{1}}

\newcommand{\R}{\mathbb{R}}
\newcommand{\Z}{\mathbb{Z}}
\newcommand{\N}{\mathbb{N}}
\renewcommand{\tilde}{\widetilde}
\renewcommand{\hat}{\widehat}

\newcommand{\cT}{{\ensuremath{\mathcal T}} }

\newcommand{\bP}{{\ensuremath{\mathbf P}} }
\newcommand{\bE}{{\ensuremath{\mathbf E}} }

\DeclareMathSymbol{\leqslant}{\mathalpha}{AMSa}{"36} 
\DeclareMathSymbol{\geqslant}{\mathalpha}{AMSa}{"3E} 
\DeclareMathSymbol{\eset}{\mathalpha}{AMSb}{"3F}     
\renewcommand{\leq}{\;\leqslant\;}                   
\renewcommand{\geq}{\;\geqslant\;}                   
 
\newcommand{\sumtwo}[2]{\sum_{\substack{#1 \\ #2}}} 

\newcommand{\pin}{\ensuremath{\pi_n}}

\newcommand{\T}{\ensuremath{\mathcal{T}}}

\newcommand{\bbE}{{\ensuremath{\mathbb E}} }

\newcommand{\bbN}{{\ensuremath{\mathbb N}} }

\newcommand{\bbP}{{\ensuremath{\mathbb P}} }

\newcommand{\bbZ}{{\ensuremath{\mathbb Z}} }

\renewcommand{\epsilon}{\varepsilon}

\newcommand{\gb}{\beta}
\newcommand{\gep}{\varepsilon}       

\newcommand{\gO}{\Omega}
\newcommand{\gl}{\lambda}

\makeatletter
\def\captionfont@{\footnotesize}
\def\captionheadfont@{\scshape}

\long\def\@makecaption#1#2{%
  \vspace{2mm}
  \setbox\@tempboxa\vbox{\color@setgroup
    \advance\hsize-6pc\noindent
    \captionfont@\captionheadfont@#1\@xp\@ifnotempty\@xp
        {\@cdr#2\@nil}{.\captionfont@\upshape\enspace#2}%
    \unskip\kern-6pc\par
    \global\setbox\@ne\lastbox\color@endgroup}%
  \ifhbox\@ne 
    \setbox\@ne\hbox{\unhbox\@ne\unskip\unskip\unpenalty\unkern}%
  \fi
  \ifdim\wd\@tempboxa=\z@ 
    \setbox\@ne\hbox to\columnwidth{\hss\kern-6pc\box\@ne\hss}%
  \else 
    \setbox\@ne\vbox{\unvbox\@tempboxa\parskip\z@skip
        \noindent\unhbox\@ne\advance\hsize-6pc\par}%
\fi
  \ifnum\@tempcnta<64 
    \addvspace\abovecaptionskip
    \moveright 3pc\box\@ne
  \else 
    \moveright 3pc\box\@ne
    \nobreak
    \vskip\belowcaptionskip
  \fi
\relax
}
\makeatother
\def\writefig#1 #2 #3 {\rlap{\kern #1 truecm
\raise #2 truecm \hbox{#3}}}

\renewcommand{\Pr}{ \mathrm P}
\newcommand{ \rel}{ t_{\mathrm{rel}} }
\newcommand{ \reln}{ t_{\mathrm{rel}}^{(n)} }
\newcommand{ \mix}{ t_{\mathrm{mix}} }
\newcommand{ \mixn}{ t_{\mathrm{mix}}^{(n)} }
\newcommand{ \dsep}{ d_{\mathrm{sep}} }
\newcommand{ \dsepn}{ d_{\mathrm{sep}}^{(n)} }
\newcommand{ \sep}{ t_{\mathrm{sep}} }
\newcommand{ \sepn}{ t_{\mathrm{sep}}^{(n)} }
\newcommand{ \sepeps}{ t_{\mathrm{sep}}(\epsilon) }
\newcommand{ \sepneps}{ t_{\mathrm{sep}}^{(n)}(\epsilon) }

\newcommand{ \mixneps}{ t_{\mathrm{mix}}^{(n)}(\epsilon) }

\newcommand{ \TV}{ \mathrm{TV} }

\newcommand{ \cL}{ \mathcal L }

\newcommand{\eps}{\epsilon }

\newcommand{\tf}{\textsc{f}}

\usepackage[normalem]{ulem}

\begin{document}

\title[Separation and Total Variation Cutoffs]{Total Variation and Separation Cutoffs are not equivalent and neither one implies the other}
\author{ Jonathan Hermon
\and Hubert Lacoin
\and
Yuval Peres
}

\date{}

\begin{abstract}
The cutoff phenomenon describes the case when an abrupt transition occurs in the convergence of a Markov chain to its equilibrium measure.
There are various metrics which can be used to measure the distance to equilibrium, each of which corresponding to a different notion of cutoff.
The most commonly used are the total-variation and the separation distances.
In this note we prove that the cutoff for these two distances are not equivalent by 
constructing several counterexamples which display cutoff in total-variation but not in separation and with the opposite behavior, including 
lazy simple random walk on a sequence of uniformly bounded degree expander graphs.
These examples give a negative answer to a question of Ding, Lubetzky and Peres.

\end{abstract}

\maketitle

\section{Introduction}

Consider an irreducible discrete-time Markov chains $X=(X_{t})_{t\ge 0}$, defined on a finite state space $\gO$ 
(we call a chain finite if $\Omega$ is finite).
We let $P$  denote its transition matrix. We further assume that $X$ is reversible, that is that there exists a  probability measure $\pi$
which satisfies the detailed balanced equation
\begin{equation*}
\forall x,y \in \Omega, \quad \pi(x)P(x,y)=\pi(y)P(y,x).
\end{equation*}
This measure is unique because of irreducibility.
Let us assume furthermore that our Markov chain is lazy, meaning that  
\begin{equation}
 \forall x \in \gO,\quad  P(x,x) \ge 1/2.
\end{equation}
A particular important case of such a Markov chain is \textbf{{\em lazy simple random walk}}
(SRW) on a simple graph  $G=(V,E)$, in which case $\gO=V$, $P(x,y)=\frac{1_{{\{ x=y \}}}}{2}+\frac{1_{\left\{\{x,y\} \in E\right\}}}{2 \deg(x)}$ and $\pi(x)=\frac{\deg(x)}{2|E|}$, where $\deg(x):=|\{y:\{x,y\} \in E \}|$
 and $|\cdot|$ denotes the cardinality of a set.

\medskip

It is a classic result of probability theory that for any initial condition the distribution 
of $X(t)$ converges to $\pi$ when $t$ tends to infinity. The object of the theory of Mixing times of Markov chains is to
study the characteristic of this convergence (see \cite{cf:LPW} for a self-contained introduction to the subject).



We denote by $\Pr_{x}^t$ ($\Pr_{x}$) the distribution of $X_t$ (resp.~$(X_t)_{t
\ge 0}$), given that $X_0=x$.   For any two distributions $\mu,\nu$ on $\Omega$,  their \emph{\textbf{total-variation
distance}} is defined to
be
\begin{equation}\label{def:TV}
\|\mu-\nu\|_\mathrm{TV} := \frac{1}{2}\sum_{x\in \gO } |\mu(x)-\nu(x)|=  \!\!\! \!\!\!  \!\!\!   \sum_{\{ x \ : \ 
 \mu(x)>\nu(x)\}} \!\!\!  \!\!\!  \!\!\!  \mu(x)-\nu(x)=1-\sum_{x\in \gO} \min ( \mu(x), \nu(x)).
\end{equation}
The worst-case total-variation distance at time $t$ is defined as 
\begin{equation}\label{def:dt}
d(t):= \max_{x \in \Omega} d_{x}(t), \text{ where } d_{x}(t):=  \| \Pr_x(X_t \in \cdot)- \pi\|_\mathrm{TV}.
\end{equation}
The (total-variation) $\gep$\emph{\textbf{-mixing-time}} is defined as  $$t_{\mathrm{mix}}(\epsilon)
:= \inf \left\{t : d(t) \leq
\epsilon \right\}. $$ 
Similarly, the (worst-case) \emph{\textbf{separation distance}} from stationarity at time $t$ is defined as 
$$\dsep(t):=1-\min_{x,y\in \gO} P^t(x,y)/\pi(y), $$
and the $\epsilon$\emph{\textbf{-separation-time}} (the ``$\epsilon$ separation-mixing-time") is defined as
$$\sepeps:= \inf \left\{t : \dsep(t) \leq
\epsilon \right\}. $$

When $\epsilon=1/4$ we omit it from the above
notation.

\medskip

Next, consider a sequence of chains, $((\Omega_n,P_n,\pi_n): n \in \N)$,
each with its corresponding
worst-distances from stationarity $d^{(n)}(t)$, $\dsepn(t)$, its mixing and separation times $t_{\mathrm{mix}}^{(n)}$, $\sepn $,
etc.. Loosely speaking, the total-variation (resp.~separation) \emph{\textbf{cutoff phenomenon}}
 is said to occur when over a negligible period of time, known as the \emph{\textbf{cutoff
window}}, the worst-case total variation distance (resp.~separation distance) drops abruptly from a value
close to 1 to near $0$. In other words, one should run the $n$-th chain until
time $(1-o(1))\mixn $ (resp.~$(1-o(1))\sepn $) for it to even slightly mix in total variation (resp.~separation), whereas
running it any further after time $(1+o(1))\mixn $ (resp.~$(1+o(1))\sepn $) is essentially redundant.
Formally, we say that the sequence exhibits a \emph{\textbf{total-variation cutoff}} (resp.~\emph{\textbf{separation cutoff}}) if the
following
sharp transition in its convergence to stationarity occurs:
\begin{equation}\label{def:cutoff}
\forall \gep\in (0,1/2], \quad  \lim_{n \to \infty}t_{\mathrm{mix}}^{(n)}(\epsilon)/t_{\mathrm{mix}}^{(n)}(1-\epsilon)=1 
\left( \text{ resp. }\lim_{n \to \infty}t_{\mathrm{sep}}^{(n)}(\epsilon)/t_{\mathrm{sep}}^{(n)}(1-\epsilon)=1\right).
\end{equation}

\medskip

It is a classical result (e.g.\  \cite[Lemmas 6.13 and 19.3]{cf:LPW} or \eqref{mixingbndonsep}) that under reversibility the separation and total-variation distances and mixing times 
can be compared as follows (the second line being an easy consequence of the first)
\begin{equation}
\label{eq:TVsepasymeqiv}
\begin{split}
& \forall t\ge 0, \quad d(t)\le \dsep(t)\le 1-(1-\min(2d(t/2),1))^2\le 4 d(t/2), \\ & \forall a \in (0,1), \quad \quad  \mix(a) \le \sep(a)  \le 2\mix(a/4).
\end{split}
\end{equation}
Another important family of distances  is the family of $\ell_p$ distances ($1 \le p \le \infty$):
\begin{equation*}
\label{eq: Lpdef}
\|\mu-\nu \|_{p,\pi}:=\begin{cases}\left[\sum_{x} \pi(x) a_{\mu,\nu,\pi}^p(x) \right]^{1/p}, & 1 \le p<\infty, \\
\max_{x\in \gO} a_{\mu,\nu,\pi}(x), & p = \infty, \\
\end{cases}
\end{equation*}
where $a_{\mu,\nu,\pi}(x):=|\mu(x)-\nu(x)|/\pi(x) $ (observe that $\|\mu-\nu \|_{1,\pi}=2\|\mu - \nu \|_{\mathrm{TV}} $). Note that the notion of distance to equilibrium and mixing time can be transposed to these distances by replacing $\|\cdot \|_{\mathrm{TV}}$ by $\|\cdot \|_{p,\pi}$
in \eqref{def:dt}. For $a\in (0,\infty)$ we denote the $a$-th $\ell_p$-mixing time by $t_{\ell_p}(a)$.
Under reversibility, the $\ell_p$ distances can be compared as follows (see \cite[Proposition 5.1]{cf:Chen})
\begin{equation}
\label{eq: 1.7}
 \begin{split}
 t_{\ell_2}(a) &\le  t_{\ell_p}(a) \le   2 t_{\ell_2}(\sqrt{a}) \quad \text{ for } p\in (2,\infty],\\
 \frac{1}{m_p} t_{\ell_2}(a^{m_p}) &\le  t_{\ell_p}(a)\le   t_{\ell_2}(a)  \quad \quad \, \, \text{ for } p\in (1,2),
  \end{split}
\end{equation}
where $m_p:= \lceil p/(2(p-1)) \rceil$.
Hence in some sense, up to a multiplicative constant, the different $\ell_p$  mixing times ($p\in (1,\infty]$)  are equivalent. 
It turns out that under reversibility the notion of cutoff for these distances are also equivalent.

\begin{atheorem}[Chen and Saloff-Coste \cite{cf:Chen}]
\label{thm: Lpcutoffequiv} Let $(\Omega_n,P_n,\pi_n)$ be a sequence of reversible lazy Markov chains. Let  $\lambda_2^{(n)} $ be the second largest eigenvalue of $P_n$.
 Then the following assertions are equivalent
 \begin{itemize}
 \item
The sequence exhibits $\ell_p$-cutoff for some $1<p \le \infty$.
\item
The sequence exhibits $\ell_p$-cutoff for all $1<p \le \infty$.
\item $\lim_{n \to \infty}(1-\lambda_2^{(n)}) t_{\ell_2}^{(n)}(1/2)=\infty $.
\end{itemize} 
\end{atheorem}

Observe that under reversibility (for any fixed chain) \eqref{eq:TVsepasymeqiv} expresses  an equivalence between the separation and the total-variation mixing times,
parallel to the one, expressed in \eqref{eq: 1.7}, holding between the different $\ell_p$  mixing times for $p\in (1,\infty]$.
Hence a natural question (in light of Theorem \ref{thm: Lpcutoffequiv}) is
whether (under reversibility) there is cutoff in total-variation if and only if there is cutoff in separation. This is Question 5.1 in 
\cite{cf:bdcutoff}, where an affirmative answer was given for the class of birth and death chains (which are Markov chains for which the set of edges $(x,y)$ 
with $P(x,y)>0$ forms a segment).
In fact, both cutoffs were shown to be equivalent to the {\em product condition}
\eqref{eq:product}.
\begin{atheorem}[Ding, Lubetzky and Peres  \cite{cf:bdcutoff},  Diaconis and Saloff-Coste \cite{cf:DSC06}]
\label{thm: bdcase}
A sequence of birth and death chains exhibits total variation cutoff iff it exhibits separation cutoff. 
\end{atheorem}

In this note we give a negative answer to that question in general by constructing counter-examples.

\begin{theorem}\label{maintheorem}

 \begin{itemize}
 \item[(i)] Total-variation and separation cutoff are not equivalent for lazy reversible Markov chains and
 neither one implies the other.
 \item[(ii)] The above statement remains true within the class of lazy simple random walks on graphs of maximal degree at most 7.
 \end{itemize}
\end{theorem}

\begin{remark}
We can also produce non-reversible or non-lazy counter-examples by performing artificial modifications in our chains, but this is not a very 
important point. Non-lazy or non-reversible chains can have very pathological behavior and we want to 
underline that we are not using ``unfair tricks" to produce our counter-examples.
\end{remark}

Of course a full proof of this statement only requires two counter-examples as (ii) is a stronger statement than (i).
However, we have chosen to include also examples that are not simple random-walks because they are much simpler.
We present a total of five counter-examples.
Apart from the first one,
they are all lazy (weighted nearest-neighbor) random walks on bounded degree graphs, with transition rates which are bounded away from zero.
The last two example, which are a bit more technical to analyze, are lazy SRWs on a sequence of bounded degree graphs 
$G_n:=(V_n,E_n)$ (i.e.~$\sup_n \max_{v \in V_n}\deg (v) < \infty$).

\medskip

Note that for all our counter-examples 
the graph supporting the transitions contains some cycles. An interesting open problem is to determine whether Theorem \ref{thm: bdcase} can extended to the case of lazy weighted nearest-neighbor random walk on trees 
for which it is already known (cf.~\cite{cf:Basu}) that separation cutoff implies total-variation cutoff.

%
%

\medskip

A sequence of Markov chains is said to display  \emph{\textbf{pre-cutoff}} (in total-variation resp. separation) if 
\begin{equation*}
 \sup_{0<\epsilon<1/2} \limsup_{n\to \infty} \mixneps / \mixn(1-\eps)< \infty \text{ resp. } \sup_{0<\epsilon<1/2} \limsup_{n\to
\infty} \sepneps / \sepn(1-\eps)< \infty.
\end{equation*}
We call the value of the $\sup$ above the \emph{\textbf{pre-cutoff ratio}}.
 Equation \eqref{eq:TVsepasymeqiv} implies that
\begin{equation}\label{groto}
 \sup_{\gep \in (0,1/2]}\limsup_{n\to \infty} \sepn(\eps)/\sepn (1-\eps) \le 2  \sup_{\gep \in (0,1/2]} \limsup_{n\to \infty} \mixneps /\mixn(1-\eps).
\end{equation}
A symmetrized version of this inequality also holds provided that $\mixn$ goes to infinity 
(the assumption being present just to avoid pathological behavior)
\begin{equation}\label{groto2}
 \sup_{\gep \in (0,1/2]} \limsup _{n\to \infty} \mixneps / \mixn(1-\eps)\le 2\sup_{\gep \in (0,1/2]}\limsup_{n\to \infty} \sepn(\eps)/\sepn (1-\eps).
\end{equation}
 The proof of \eqref{groto2} involves more computation than \eqref{groto}. We present a complete proof of it in Appendix 
\ref{sec: groto2})

These two inequalities imply that the notion of pre-cutoff is equivalent for the two distances and the pre-cutoff ratio of one is at most twice 
that of the other. In particular, cutoff in one distance implies pre-cutoff with ratio at most $2$ in the other.
With our examples, we shall show that this is in fact sharp in some cases:

\begin{remark}\label{rem:ration}
There exists a sequence of lazy reversible Markov chains for which we have cutoff in total-variation and only pre-cutoff with ratio $2$ in separation and vice-versa.
\end{remark}

\medskip

Our last point of comparison between total-variation mixing and separation mixing is related to the width of the
\emph{\textbf{cutoff window}}. We say that a sequence of chains exhibits total-variation (resp.~separation) cutoff with a cutoff window 
$w_n$ if $w_n=o(\mixn)$ and for all $0<\gep \le 1/4$ there exists some constant $C_{\gep}>0$ (depending only on $\gep$) such that 
$$\forall n, \quad \mixn(\gep)-\mixn(1-\gep) \le C_{\gep}w_n \quad (\text{resp. }\sepn(\gep)-\sepn(1-\gep) \le C_{\gep}w_n).$$
Note that the window defined in this manner is not unique, but informally ``the'' cutoff window is given by the ``smallest such $w_n$''.
Our examples demonstrate that the cutoff windows for total-variation and separation do not have the same behavior. 
 
The following result is due to Chen and Saloff-Coste \cite[Theorem 3.4]{cf: Chen2}. We present a much simpler proof in the Appendix.
\begin{atheorem}
\label{thm: window}
Let $(\Omega_n,P_n,\pi_n)$ be a sequence of lazy irreducible finite chains which exhibits total-variation cutoff with a cutoff window $w_n$. 
Then $w_n=\Omega (\sqrt{\mixn})$.
\end{atheorem}

The bound given by Theorem \ref{thm: window} is obviously sharp for the biased random walk on a segment. 
Conversely, some very standard Markov chains like the lazy SRW on the $n$-dimensional hyper-cube have a cutoff window $w_n>> \sqrt{\mixn}$ (here $w_n=n$ and $\mixn=(\frac{1}{2}\pm o(1) )n\log n$). As indicated in Remark \ref{rem: 1.6} the laziness assumption in Theorem \ref{thm: window}  can be replaced by the assumption that $\inf_{n} \min_{x \in \Omega_n} P_n^2(x,x)>0$ (as is the case for simple random walk on a sequence of bounded degree graphs). 


\medskip

In light of Theorem \ref{thm: window} one might expect that whenever separation cutoff occurs for a sequence of discrete-time lazy chains, the width of the separation cutoff window is $\Omega(\sqrt{\sepn})$. We are unaware of any previously analyzed example in which this fails. We find it remarkable that as the following remark asserts, the width of the separation cutoff window for a sequence of discrete-time lazy SRWs on a sequence of bounded degree graphs, can in fact be a constant! This, or more precisely, the mechanism that allows such behavior (see \S~\ref{sec: idea2} for more on this point) demonstrates that the separation distance can exhibit profoundly different behaviors than the total variation distance.

Our counter-examples show that the cutoff window in one distance can be as small as allowed even if there is no cutoff for the other distance:


\begin{remark}
\label{prop: window}
We will construct sequences of bounded degree graphs such that the corresponding sequences of lazy SRWs exhibit the following behaviors (resp.) 
\begin{itemize}
\item[(i)] There is no separation cutoff but there is total-variation cutoff with window $\sqrt{\mixn}$.
\item[(ii)] There is no total-variation cutoff but there is separation cutoff with window $1$.
\end{itemize}
In \S~\ref{sec: idea2}  we refine the statement of (ii) and describe further surprising properties of the relevant example for (ii) above (listed in  \S~\ref{sec: idea2}   as properties $(i)$-$(v)$).
\end{remark}

\begin{remark}
Let $\delta_{n} \in (0,1)$. We call a sequence of discrete time chains $(\Omega_{n},P_{n},\pin)$,  $\delta_n$-lazy if for all $n$, $P_n(x,x)\ge \delta_n$ for all $x \in \Omega_n$. 
It is not hard to extend the proof of Theorem \ref{thm: window} and show that if a sequence of $\delta_n$-lazy chains exhibits total-variation cutoff with a
window $w_n$, then $w_n=\Omega (\sqrt{\delta_n(1-\delta_n) \mixn})$. 

Theorem \ref{thm: window} can also be extended to the continuous time setup, with the additional assumption that the sum of the transition rates 
from any given state is bounded above by 1 (or by some absolute constant).
\end{remark}
\begin{remark}
\label{rem: 1.6}
Let $G_n=(V_n,E_n)$ be a sequence of connected non-bipartite simple graphs of maximal degree $d_n$. Consider the sequence of (non-lazy) SRWs on $G_n$. Then $P_n^2(v,v) \ge 1/d_n$, for every $v \in V_n$. By considering $P^2$ rather than $P$ it follows from the previous remark that if the sequence exhibits total-variation cutoff with a window $w_n$, then $w_n=\Omega (\sqrt{\mixn/d_n})$. This is in fact sharp by considering a sequence of random $d_n$-regular graphs of size $n$ for some $d_n$ such that $\lim_{n \to \infty}d_n = \infty$ and $d_n=o(\frac{\log n}{\log \log n})$ \cite[Theorem 3]{cf:LS3}. 
\end{remark}

%

%
%
%

\subsection{Organization of the note.}

In \S~\ref{sec: overview} we describe the construction of our examples and our general strategy. 
We also describe relevant examples due to Aldous and Pak.  

In \S~\ref{sec: Pre} we introduce a general framework, which under a certain condition, allows to reduce the study of the mixing-time to the study of 
the hitting time of a special point.

In \S~\ref{sec: TVcutoff} 
we describe two examples of sequences of Markov chains which exhibit total-variation cutoff 
but do not exhibit separation cutoff. The first example, Example \ref{ex: 1}, demonstrates that \eqref{groto} may be sharp (even when the r.h.s.~of \eqref{groto} equals 1). 
The second example, Example \ref{ex: 2}, is a weighted nearest neighbor random walk on a bounded degree graph with transition probabilities which are bounded away from 0 and 1. 

In \S~\ref{sec: Sepcutoff} we construct an example of a sequence of Markov chains that exhibits separation cutoff but no total-variation cutoff 
(Example \ref{ex: 3}).

Finally, in \S~\ref{sec: LSRW} we transform Examples \ref{ex: 2} and \ref{ex: 3} into examples of sequences of lazy SRWs on bounded degree Expander graphs. The reason we first describe Examples \ref{ex:
2} and \ref{ex: 3} is that the key ideas of our constructions are more transparent in theses examples.
\section{An overview of the main ideas of our constructions}\label{sec: overview}

\subsection{A very basic chain with different cutoff times for separation and total variation}

In this section we settle with a high-level description of some key ideas. 
Let us first present a very simple Markov chain which exhibits cutoff in both distances (see Figure \ref{fig:basicchain}) but for which the mixing-time in separation is twice as large as
that in total variation.

\medskip

Consider a random walk on a segment $a,b$ of length $2n$ which presents a constant bias towards the middle point which we call $z$ (see Figure \ref{fig:basicchain}).
Most of the equilibrium measure is concentrated on a small neighborhood of $z$ and for this reason (cf.\ Proposition \ref{prop: crittv}) the total-variation mixing-time corresponds to the time which is needed to hit $z$ (starting from either of the end-points). The system displays cutoff because this hitting time is concentrated around its mean.

\begin{figure}[h]
\begin{center}
\leavevmode
\epsfxsize =12 cm
\psfragscanon
\psfrag{a}{\tiny $a$}
\psfrag{b}{\tiny $b$}
\psfrag{z}{\tiny $z$}
\psfrag{1/3}{\tiny $1/3$}
\psfrag{1/6}{\tiny $1/6$}
\psfrag{1/4}{\tiny $1/4$}
\psfrag{1/2}{\tiny $1/2$}
\epsfbox{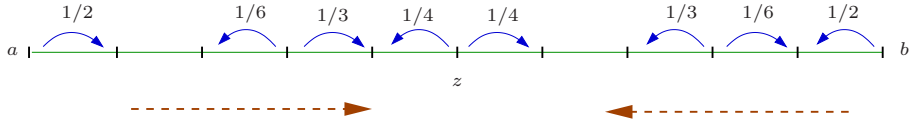}
\end{center}
\caption{\label{fig:basicchain} 
A very simple chain for which the separation mixing-time is twice as large as the total-variation mixing-time ($6n$ and $12n$, respectively).
The transition rates (apart from at the special states $a$, $b$ and $z$) are $1/3$ in the $z$ direction and $1/6$ in the opposite one (the holding probability is $1/2$), making the 
chain travel at speed $1/6$ towards $z$.}
\end{figure}
\medskip

The separation mixing-time on the other hand is twice as large. Roughly speaking, this is because for $P^t(a,b)$ to come close to its equilibrium value, 
``information" has to 
pass from one end to the other. The time required for this to occur corresponds more or less to the sum of the times needed to reach $z$ from $a$ and $b$, respectively 
(see Proposition \ref{prop: critsep}).

\medskip

This scheme with two extremal opposite initial conditions, though not ubiquitous among Markov chains, appears in many natural examples for which cutoff has been proved: e.g.\ 
the lazy SRW on the hyper-cube (see \cite[Theorem 18.3]{cf:LPW}),
the Ising model at high temperature \cite{cf:LS2} or the adjacent-transposition shuffle on the segment \cite{cf:Lac}.

\subsection{An idea to avoid cutoff in separation while keeping that in total-variation}
\label{s: idea1}

Our idea to produce counter-examples with total-variation cutoff but only pre-cutoff in separation 
is to modify the structure (state space and transition rates) of the simple chain above (Figure \ref{fig:basicchain}),  only on one side (say, the side of $b$), 
to break the symmetry. To be precise, in Example \ref{ex: 2} we first set the holding probabilities on both sides to be $3/4$ (and consider the obtained chain as the ``original chain", as opposed to Example \ref{ex: 1}, for which the chain in Figure \ref{fig:basicchain} serves as the ``original chain") before modifying the $b$-side. We want to perform our modifications in the following manner:
\begin{itemize}
 \item We want to keep the property that every path from $a$ to $b$ goes through $z$, which shall still bear a positive proportion of the 
 equilibrium mass.
 \item We want $a$ to remain the initial condition from which it takes the longest time to reach equilibrium (equivalently, to hit $z$). More precisely,  we want that also after the modification, the distribution of the hitting time of $z$, $T_z:= \inf\{ t \ :X_t=z\}$, starting from $a$ would still stochastically dominate the distribution of $T_z$, starting from any other initial state. Moreover, we want the hitting time distribution of $z$, starting from any state between $a$ and $z$ (including $a$), to remain un-changed.
 \item We want the hitting time  of $z$ from initial state $b$, to become non-concentrated, and to remain of the same order of magnitude as
 the mixing-time of the whole chain. Moreover, we want this hitting time to remain (stochastically) larger than the hitting time  of $z$, starting from any other state which lies between $b$ and $z$, and to become stochastically dominated by the hitting time distribution of $z$ (from $b$) in the original chain (which equals the hitting time distribution from $a$ in the modified chain).
 \end{itemize} 
In this manner, the hitting time distribution of $z$ under $\Pr_a$ remains un-changed (and in particular, remains concentrated). Moreover, after the  modification it is still the case that $d(t) \approx \Pr_a[T_z>t]$,  and thus by the aforementioned concentration there is still cutoff in total-variation (see Proposition \ref{prop: crittv}).   Using Proposition \ref{prop: critsep},
we deduce that $\dsep(\mix+t) \approx \Pr_b[T_z>t] $ and so there is no cutoff in separation as the hitting time distribution of $z$ under $\Pr_b$ in the modified chain is no longer concentrated. 

To perform such a modification, we borrow ideas from previous constructions of Pak (for Example \ref{ex: 1}) and Aldous (for Example \ref{ex: 2}), which we present now.

\subsection{Related Constructions}

When the product condition (Definition \ref{def: productcondition}) was shown to be a necessary condition for cutoff, it was conjectured 
that it should also be a sufficient one for ``nice" chains. However, two counter-examples constructed, respectively by Aldous and Pak 
(see \cite[Example 8.1]{cf:Basu}, \cite{cf:Chen} and  \cite[Chapter 18]{cf:LPW}  for a more detailed description and analysis), show that in general the product condition does not imply cutoff. 
The mechanisms used  to prevent cutoff in those two constructions are of different nature.
\begin{itemize}
\item Aldous' example (Figure \ref{fig:aldous}) locally looks like a biased random walk on a segment, 
so that most of the equilibrium measure is concentred on a small neighborhood of the end-point towards which the walk is biased (we call this end of the segment  $z$  and the opposite one $b$).
To avoid cutoff, the half of the segment closer to $z$ is split into two distinct parallel branches. The transition rates on these branches are tuned so that 
there is still a bias towards $z$ but such that one path is slower than the other.
Starting furthest away from equilibrium (i.e.\ at state $b$) we have two possible scenarios to reach $z$ given by the two distinct branches and  
the probability of each is bounded away from $0$ and $1$.
As the speed along the two branches is different, the CDF of the hitting time distribution of $z$ starting from $b$ has two abrupt jumps. Consequently, $d^{(n)}(t)$ exhibits two distinct abrupt drops and there is no cutoff. 
\item  Pak's idea is to start with a sequence of chains which exhibits cutoff and to modify it by adding transitions
which are such that with a constant rate (which is chosen to be somewhere between the spectral gap and the inverse of the 
mixing-time of the original chain, say their geometric mean) the system is brought to equilibrium at once.
For the modified Markov chain, the total-variation distance decays (up to a negligible error) exponentially 
with the rate of the newly added transitions and hence cutoff does not occur, neither pre-cutoff.
\end{itemize}

\begin{figure}[h]
\begin{center}
\leavevmode
\epsfxsize =10 cm
\psfragscanon
\psfrag{a}{\tiny $a$}
\psfrag{b}{\tiny $b$}
\psfrag{z}{\tiny $z$}
\psfrag{1/3}{\tiny $1/3$}
\psfrag{1/6}{\tiny $1/6$}
\psfrag{1/12}{\tiny $1/12$}
\psfrag{1/8}{\tiny $1/8$}
\psfrag{1/4}{\tiny $1/4$}
\epsfbox{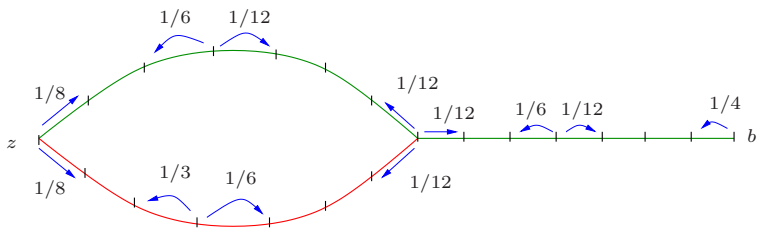}
\end{center}
\caption{\label{fig:aldous} 
A version of Aldous' example. The walk is always biased towards $z$ but the speed of the walk depends on the branch.
On the top branch, as well as on the rest of the segment, the transition rates are $1/6$ in the $z$ direction and $1/12$ in the opposite one
(the holding probability is $3/4$) whereas on the bottom branch the (exit) rates are twice as large (and the holding probability is $1/2$), resulting in a larger speed. As a result, two transitions occur for the total-variation distance 
at times $9n$ and $12n$ respectively, where $n$ denotes the total distance from $z$ to $b$ and the length of each of the two parallel branches is $\lceil n/2 \rceil $ (above $n=14$).
The rates at $b$, $z$ and at the branching point are not very relevant but we display them for the sake of concreteness.}
\end{figure}

\medskip

In our Example \ref{ex: 1} (see Figure \ref{fig:ex1}), we adapt Pak's idea: on the $b$-side (of the chain from Figure \ref{fig:basicchain}) we add transitions from states on the $b$-side to the center of mass $z$, 
and we choose the inverse of the rate to be of the same order as the mixing-time (which is of order of the length of the segment: $n$). This makes the hitting time of $z$ started from $b$ non-concentrated and (stochastically) smaller than 
started from $a$. Moreover, after this modification, all of the properties described in the beginning of \S~\ref{s: idea1} are satisfied.

\medskip

In our Example \ref{ex: 2}, (see Figure \ref{fig:ex2}), we simply replace the $b$-side by Aldous' construction, and set the holding probability on the $a$-side to be $3/4$ (which is the holding probability of the slow branch of  the $b$-side). After this modification, all of the properties described in the beginning of \S~\ref{s: idea1} are satisfied.

\subsection{An idea to keep cutoff in separation while avoiding that in total-variation}
\label{sec: idea2}

For this part we must rely on a different idea. What we want to alter in our chain is the way the separation distance shrinks to zero. Loosely speaking,
in the original chain on the segment, the separation mixing-time is determined by the sum of the hitting times of $z$ from $a$ and $b$ since
$z$ is the only channel of communication between the two extremities.

\medskip

Our construction (Example \ref{ex: 3}) relies on the following idea (see Figure \ref{fig:ex3}). We take the length of the line segment to be $2(M+1)n$ for some large (fixed) integer $M$.
\begin{itemize}
 \item We connect the two sides of the segment at a second point $z'$ which is far from the center of mass $z$. We do so by merging the two states which are of distance $n$ from $z$ (one on the $a$-side and one on the $b$-side) into a single state $z'$. This connection maintains the cutoff in separation.   However, it has the effect 
 of  shortening the separation cutoff time by some constant factor, while, as we now describe, drastically altering the nature of the abrupt transition of $\dsepn(t)$ around the (separation) cutoff time.
It follows from our analysis of Example \ref{ex: 3} and the refined analysis of Example  \ref{ex: 5} in \S~\ref{sec: Rem 1.6}, that provided that $M$ is taken to be sufficiently large:
\begin{itemize}
\item[(i)] Also after creating the connection at $z'$ we have that  $$\lim_{n \to \infty} \sup_t |\dsepn(t) - \max (0,1-P_n^t(a,b)/\pin(b))| =0.$$ 
\item[(ii)] Due to the connection of $A$ and $B$ at point $z'$, up to negligible terms, around the separation cutoff time, $P_n^{t}(a,b)$ is supported by trajectories which never get much closer to $z$ than $z'$ is, and so are contained in a set whose stationary probability is exponentially small in $n$.
\item[(iii)]
Let $T_{z'}^{a,b}$ (Definition \ref{def: convolution}) be a random variable distributed as a convolution of the hitting time distribution of $z'$ started from $a$ with that started from $b$ (in this case the two distributions are identical). Around the (separation) cutoff time, $P_n^t(a,b)/\pin(b)$ can be understood in terms of the behavior of  $T_{z'}^{a,b}$ in the large deviation regime (namely, the cutoff occurs around the time $t$ for which $\mathbb{P}[T_{z'}^{a,b} \ge t] \approx \pi_n(z')=\Theta(2^{-n}) $). 
\item[(iv)]
Around $\sepn$, $P_n^{t}(a,b)/\pin(b)$ grows exponentially in $t-\sepn$, for $t \ge \sepn$  (and decays exponentially for $t< \sepn $) and continues to do so for $\Theta(n)$ steps around $\sepn$
 (in particular, shortly after $\sepn$, $(a,b)$ no longer minimizes $P_n^t(x,y)/\pin(y) $). By (i), it follows that $w_n = 1$ is a (separation) cutoff window (and we can take $C_{\gep}=C |\log \gep| $, for some absolute constant $C$, for all $\gep \in (0,1/4]$).
\item[(v)]
 $\sup_{t} P_n^t(a,b)/\pin(b)=\Theta(\max_{t}\mathbb{P}[T_{z'}^{a,b} = t]/ \pi_n(z') )=\Theta(2^n/n) \to \infty $ as $n \to \infty$.
\end{itemize}

This behavior (namely, on the one hand having property ($i$) and on the other having properties ($ii$), ($iv$) and ($v$)) is atypical and quite surprising at first sight.

\medskip

We are not done yet, as after creating the connection at $z'$, there are two symmetric parallel distinct branches from $z'$ to the center of mass $z$, resulting in the hitting time of $z$ from either $a$ or $b$ being concentrated. Consequently, there is still cutoff in total-variation (as by Proposition \ref{prop: crittv}, $d(t) \approx \Pr_a[T_z>t]=\Pr_b[T_z>t]$). 
 
 \item   We break the symmetry (between the two branches, but not between $a$ and $b$)  in order to ``destroy" the cutoff in total-variation by making the speed along the two paths which link $z'$ to $z$ different
 as in Aldous' example (Figure \ref{fig:aldous}). Observe that as opposed to Examples \ref{ex: 1}-\ref{ex: 2}, here $a$ and $b$ play symmetric roles (the chain looks the same starting from either one of them).
 \end{itemize}
As one should expect from property ($ii$) above (provided that $M$ is sufficiently large), breaking the symmetry as described above does not influence the asymptotic
pattern of convergence in separation, and ($i$)-($v$) above remain valid.
However the quantitative analysis of this example turns out to be more intricate than that of the first two.

%
%
%

\subsection{Constructing counter-examples which are lazy SRW on bounded degree graphs}

It was observed by Peres and Wilson that the sequence of chains in Aldous'
example could be modified into a sequence of lazy SRWs on bounded
degree expander graphs (see Definition \ref{def: Cheeger}). In \cite{cf:LS} Lubetzky and Sly constructed explicit 3-regular expanders with total-variation cutoff.

We use similar ideas to transform our Examples \ref{ex: 2}-\ref{ex: 3} into SRWs on bounded degree graphs (Examples \ref{ex: 4}-\ref{ex: 5}). Our constructions includes one new idea: by introducing a sufficient amount of symmetry, (roughly speaking) we are able to reduce the analysis of Examples \ref{ex: 4}-\ref{ex: 5} to that of Examples \ref{ex: 2}-\ref{ex: 3}. Consequently, the analysis of the asymptotic convergence profile of $d_n(t)$ is simpler than in \cite{cf:LS} (at the cost of having maximal degree $\le 7$ rather than $3$).

\section{Preliminaries}
\label{sec: Pre}

The aim of this section is to introduce some general theory which shall reduce the analysis of our Examples \ref{ex: 1}-\ref{ex: 3} 
to the analysis of hitting time distributions of a specific state. The results appearing in this section are later generalized in \S~\ref{sec: generalizations} (these generalizations reduce the analysis of Examples \ref{ex: 4}-\ref{ex: 5} to the analysis of hitting time distributions of a specific set).  All proofs are deferred to the appendix. As we shall only prove the more general versions, we now describe the correspondence between the results of this section to the ones from \S~\ref{sec: generalizations}: Proposition \ref{prop: hitcutoff} corresponds to Proposition \ref{prop: crittv}, Lemma \ref{lem: pathsdecompositions} to Lemma \ref{lem: pathsdecomposition} and Proposition \ref{cor: sepcutoffcriterion} to Proposition \ref{prop: critsep}.

\medskip

Let us first introduce some notation and standard terminology.
Recall that if $(\Omega,P,\pi)$ is a finite  irreducible reversible Markov chain,
then $P$ is self-adjoint w.r.t.~the inner product induced by $\pi$ on $\R^{\Omega}$ 
\begin{equation}
 \langle f,g\rangle_{\pi}:=\sum_{x\in \gO} \pi(x)g(x)f(x).
\end{equation}
Hence it has $|\Omega|$ real eigenvalues satisfying
 $1=\lambda_1>\lambda_2 \ge \ldots \ge
\lambda_{|\Omega|} \ge -1$ (where $\lambda_2<1$ since the chain is irreducible
and if the chain is lazy then $\lambda_{|\Omega|} \geq 0$). Define its \emph{\textbf{relaxation-time}} as $t_{\mathrm{rel}}:=(1- \max (\lambda_2,|\gl_{|\Omega|}|))^{-1}$. Note that under laziness $\rel=(1-\gl_2)^{-1}$.

\begin{definition}
\label{def: productcondition}
We say that a family of reversible Markov chains satisfies the \emph{\textbf{product condition}}
 if
 \begin{equation}\label{eq:product}
 \lim_{n\to \infty}(1- \max ( \lambda_2^{(n)},|\gl_{|\gO|}^{(n)}|))t_{\mathrm{mix}}^{(n)} = \infty \quad (\text{equivalently, }t_{\mathrm{rel}}^{(n)}=o(t_{\mathrm{mix}}^{(n)}))\, .
 \end{equation}
\end{definition}
Because of the following well-known fact (e.g.~\cite[Proposition 18.4]{cf:LPW}), all our counter-examples satisfy the 
product condition.
\begin{fact}
\label{fact: cutoffandtrel}
For a sequence of irreducible aperiodic reversible Markov chains with relaxation-times
$\{t_{\mathrm{rel}}^{(n)} \}$ and mixing-times $\{t_{\mathrm{mix}}^{(n)}
\}$, if the sequence exhibits a pre-cutoff (either in total-variation or
separation) and $\lim_{n \to \infty}t_{\mathrm{mix}}^{(n)}
=\infty $, then $t_{\mathrm{rel}}^{(n)}=o(t_{\mathrm{mix}}^{(n)})$.
\end{fact}

Given $z\in \gO$ we let 
$$T_z:= \inf\{ t \ :X_t=z\}$$ 
denote the hitting time of $z$. The following result allows us to characterize the mixing-time of the chain in terms of 
the hitting time of a given point which carries a positive proportion of the mass.
As hitting times are sometimes easier to control than mixing-times, it will assist us in determining the total-variation profile of convergence to equilibrium
in Examples \ref{ex: 1}-\ref{ex: 3}.

\begin{proposition}\label{prop: crittv}
 Let $(\gO_n,P_n,\pi_n)$ be a sequence of lazy reversible irreducible finite Markov chains which satisfies the product condition.
 Let us furthermore assume that there exists $z_n\in \gO_n$ such that 
\begin{equation}
 \inf_{n} \pi_n(z_n)>0.
 \end{equation}
Then setting 
\begin{equation}
 \tau_n(p):=\inf \left\{ t \ : \  \max_{x\in \gO_n} \Pr_x[T_{z_{n}}>t]\le p \right\},
\end{equation}
we have for any $\gep< \gep' \in (0,1)$
\begin{equation}
\label{eq: hitcutoff0}
 \limsup_{n \to \infty}\frac{\mixn(\eps')}{\tau_n(\eps)} \le 1 \text{ and } \liminf_{n \to \infty}\frac{\mixn(\eps)}{\tau_n(\eps')} \ge 1.
 \end{equation}
\end{proposition}
Note that in particular the result shows that total-variation cutoff occurs if and only if $\tau_n(\cdot)$ displays the following abrupt transition
\begin{equation}
 \forall \gep\in (0,1/2], \quad \lim_{n \to \infty}\frac{\tau_n(1-\eps)}{\tau_n(\eps)}=1.
\end{equation}

To characterize the separation time, we introduce a notion of ``double-hitting time".

\begin{definition}
\label{def: convolution}
Given $x,y$ and $z$ in $\gO$. We let $T_{z}^{x,y}$ denote a random variable obtained by taking the sum 
of two independent realizations of $T_z$, once under $\mathrm{P}_x$ and once
under $\mathrm{P}_y$. That is, $\mathbb{P}[T_{z}^{x,y}=t]:=\sum_{k =0}^t
\mathrm{P}_x[T_z=k]\mathrm{P}_y[T_z=t-k].$
\end{definition}

\begin{lemma}
\label{lem: pathsdecomposition}
Let  $(\Omega,P,\pi)$ be a finite irreducible lazy reversible Markov
chain. Consider $x,y, z
\in \Omega$.
\begin{itemize}
\item[(i)] For all $t \ge 0 $
we have that
\begin{equation}
\label{eq: septhroughz1}
P^t(x,y)/\pi(y)
 \ge \sum_{k \le t }\mathbb{P}[T_{z}^{x,y}=k]P^{t-k}(z,z)/ \pi(z)
\ge   \mathbb{P}[T_{z}^{x,y} \le t]. 
\end{equation}
In particular,
\begin{equation}
\label{eq: septhroughz3}
P^t(x,y)/\pi(y) \ge \Pr_x[T_y \le t]
\end{equation}
\item[(ii)]
If
$\mathrm{P}_{x}[T_z \le T_y]=1$ (i.e.~if every path from $x$ to $y$ goes through $z$) then for all $t \ge 0$
\begin{equation}
\begin{split}
\label{eq: septhroughz2}
& P^t(x,y)/\pi(y)
=\sum_{k \le t }\mathbb{P}[T_{z}^{x,y}=k]P^{t-k}(z,z)/ \pi(z)\\
& \le  \mathbb{P}[T_{z}^{x,y} \le t]+\frac{1}{2} t_{\mathrm{rel}} \max_{k\in \bbN} \bbP [T_{z}^{x,y} =k ] \sqrt{(1-\pi(z))/\pi(z)}.
\end{split}
\end{equation}
\end{itemize}
\end{lemma}

 All our examples would be of sequences of chains whose spectral gaps are uniformly bounded away from zero, that is, ones satisfying
\begin{equation}\label{star}
\tag{$\star$}
 \inf_{n}(1- \gl^{(n)}_2)>0.
\end{equation}

Although this is not necessary, working with such chains substantially simplifies the analysis of our examples.
To check this condition, we use the notion of the Cheeger constant and the well-known discrete analog of Cheeger's inequality \eqref{eq: Sinclair} \cite{cf:Alon1,cf:Alon2,cf:Sinclair} (the proof can also be found at \cite[Theorem
13.14]{cf:LPW}).

\begin{definition}
\label{def: Cheeger}
For any (non-empty) set $A \varsubsetneq\ \Omega $ we define 
\begin{equation*}
Q(A):=\sum_{x \in A,y \notin A }\pi(x)P(x,y) \quad \text{ and }  \quad \Phi(A):=Q(A)/\pi(A).
\end{equation*}
We define the \textbf{{\em Cheeger constant}} of the chain to be 
\begin{equation*}
 \Phi:=\min_{A:0< \pi(A) \le 1/2}\Phi(A).
\end{equation*}
We call a sequence of chains $(\Omega_n,P_n,\pin)$ an \textbf{{\em expander family}} if $\inf_n \Phi_n >0$. 
\end{definition}
The following result implies that a sequence of reversible chains satisfies \eqref{star} if and only if it is an expander family.

\begin{theorem}
\label{thm: Cheeger}
Let $\lambda_2$ be the second largest eigenvalue of a reversible transition matrix on a finite state space. Let $\Phi$ be as in Definition \ref{def: Cheeger}.  Then
\begin{equation}
\label{eq: Sinclair}
\Phi^2/2 \le 1- \lambda_2 \le 2\Phi. 
\end{equation}
\end{theorem}
It is rather straightforward to check in all of our examples that the Cheeger constant is bounded away from zero.

\begin{proposition}\label{prop: critsep}
 Let $(\gO_n,P_n,\pi_n)$ be a sequence of lazy reversible irreducible finite Markov chains which satisfies \eqref{star}.
 Let us furthermore assume that there exist  $z_n\in \gO_n$, sets $A_n, B_n \subset \gO_n$, with $A_n\cup B_n= \gO_n \setminus \{z_n\}$ and
 $a_n\in A_n$, $b_n\in B_n$,
 such that 
 \begin{itemize}
  \item [(i)] $\inf_{n} \pi_n(z_n)>0$.
  \item [(ii)] For any $x\in A_n$ and $y\in B_n$, $\Pr_x[T_{z_n}<T_y]=1$.
  \item [(iii)] For all $t$
  \begin{equation*}
    \max_{x\in A_n\cup B_n} \Pr_x[T_{z_n}>t]  =  \Pr_{a_n}[T_{z_n}>t]\quad \text{and} \quad
    \max_{y\in B_n} \Pr_y[T_{z_n}>t] =  \Pr_{b_n}[T_{z_n}>t].
\end{equation*}
  \item [(iv)]   
  \begin{equation}
\liminf_{n\to \infty} \inf_{t\ge 0} \min_{x,y\in A_n  } \left( \frac{P_n^t(x,y)}{\pi_n(y)}-\frac{P_n^t(a_n,b_n)}{\pi_n(b_n)}\right)\ge 0. 
  \end{equation}
  \item[(v)]
\begin{equation}
\lim_{n\to \infty}\max_{k\ge 0} \bP_{a_n}\left[ T_{z_n}=k\right]=0. 
\end{equation}
\end{itemize}
Then 
\begin{equation}
\lim_{n\to \infty} \sup_{t\ge 0} \left |  \dsepn(t)-\mathbb{P} [T_{z_{n}}^{a_n,b_n}>t] \right|=0. 
\end{equation}
\end{proposition}

\begin{proof}
We want
to show that  $P_{n}^t(x,y)/\pi_{n}(y)$  achieves its smallest value for $(x,y)=(a_n,b_n)$
up to a negligible correction. According to $(iv)$ we do not need to worry about the case when both $x$ and $y$ lie in $A_n$. 
For the other cases, condition $(iii)$ combined with  Lemma \ref{lem: pathsdecomposition} guaranties that 
\begin{equation}
\forall t, \, \forall (x,y) \in \gO_n^2 \setminus A_n^2, \quad  
P_{n}^t(x,y)/\pi_{n}(y)\ge  \bbP\left[ T_{z_n}^{x,y}\le t \right] \ge \bbP \left[ T_{z_n}^{a_n,b_n}\le t\right].
\end{equation}
Finally, applying Lemma \ref{lem: pathsdecomposition} again yields that 
\begin{equation}\label{framing}
0 \le \frac{P_n^t(a_n,b_n)}{\pi_n(b_n)}- \mathbb{P}[T_{z_n}^{a_n,b_n} \le t] \le  
\frac{1}{2} t_{\mathrm{rel}}^{(n)}  \max_{k\ge 0} \bbP [T_{z_n}^{a_n,b_n} =k]\sqrt{(1-\pi_{n}(z_n))/\pi_{n}(z_n)}.
\end{equation}
This allows to conclude the proof by noticing that the right-hand side of \eqref{framing} is $o(1)$ (using $(i)$ and $(v)$).
\end{proof}
\begin{remark}
\label{rem: distance}
We note that for lazy chains condition (v) in Proposition \ref{prop: critsep} follows from the condition
$\lim_{n \to \infty} \mathrm{dist}(a_n,z)=\infty $ (which is satisfied in Examples \ref{ex: 1}-\ref{ex: 3}), 
where $\mathrm{dist}(a_n,z)$ is the minimal $k$ such that $P^k(a_n,z)>0 $.  To see this, consider the non-lazy path the chain performed from $a_n$ to $z$ by time $T_z$,  $\gamma=(\gamma_0=a_n,\gamma_1,\ldots, \gamma_{\ell}=z)$  (i.e.~for all $i< \ell$, $\gamma_{i+1} \neq \gamma_i $ and possibly after spending some time at $\gamma_i$ the chain moved to $\gamma_{i+1}$). The conditional law of $T_z$, given $\gamma$, is that of a sum of $\ell$
independent geometric random variables with parameter $1/2$), and so by the local CLT its mode is at most $C/ \sqrt{ \ell} \le C/\sqrt{\mathrm{dist}(a_n,z)}$. Finally, note that the mode of a mixture is at most the maximal mode of a distribution in the mixture.   
\end{remark}

\section{Total-variation cutoff without separation cutoff examples}
\label{sec: TVcutoff}
In this section we describe two similar examples of sequences of reversible chains which exhibit total-variation cutoff but no separation cutoff.
The analysis of both examples is extremely similar. We present both examples since while the first  demonstrates that \eqref{groto} is indeed sharp, 
it is much harder to transform it into an example of lazy SRWs on bounded degree expander graphs.


\begin{example}
\label{ex: 1} Given $n\ge 2$, set $\Omega_n:=A \cup \{z\} \cup B$
where $A=A_{n}:=\{a=a_{n},a_{n-1},\ldots,a_1\}$ and $B=B_{n}:=\{b_1,b_{2},\ldots,b_{n-1},b_n=b \}$. For notational convenience we write $a_{0}:=z=:b_{0}$. 
The matrix $P_n$ has positive transition rates on the set of (un-oriented) edges 
$E=E_A \cup E_{B} \cup E_{\mathrm{Long}}$, where 
\begin{equation}
 \begin{split}
  E_A&:=\{e_k^A:=\{a_{k},a_{k-1} \} :k \in [n]\},\\
  E_{B}&:=\{e_k^B:=\{b_{k},b_{k-1},\} :k \in [n]\}, \\
  E_{\mathrm{Long}}&:=\{e_k^{\mathrm{L}}:=\{z,b_k \}:k \in [n] \}.
 \end{split}
\end{equation}
With a small abuse of notation we define $e_1^{\mathrm{L}}$ and $e_1^B$ to be two distinct parallel edges.
To each of these edges, we associate conductances (or weights), as follows
\begin{itemize}
\item
$w_n(e_k^A)=2^{-k}=w_n(e_k^B)$, for all $k \in [n]$.
\item
$w_n(e_k^{\mathrm{L}})=\frac{w_n(e_k^{\mathrm{B}})+w_n(e_{k+1}^{\mathrm{B}})}{n-1} =3 \cdot \frac{2^{-(k+1)}}{(n-1)}$ for $ k \in [n-1],$ and 
$w_n(e_n^{\mathrm{L}})=\frac{2^{-n}}{(n-1)}$. 
\end{itemize}
We let $P_n$ be the transition matrix of the $(1/2)$-lazy random walk on the graph $(\gO_n,E)$ with conductances $w_n$,
i.e.\ we set 
\begin{equation}
 \begin{split}
  P_{n}(x,x)&=1/2  \text{ for all }  x \in \Omega_n,\\
  P_{n}(x,y)&=\frac{w_n(x,y)1_{x \neq y }}{2 w_n(x)},
 \end{split}
\end{equation}
where $w_n(x):=\sum_{y\in \gO_n} w_n(x,y)$
with the convention that  $w_n(z,b_1)=w_n(e_1^{\mathrm{L}})+w_n(e_1^B)$.
This Markov chain is reversible with respect to 
$$\pi_n(x):=\frac{w_n(x)}{\sum_{y\in \gO_n}w_n(y)}.$$ 
\end{example}

\begin{figure}[h]
\begin{center}
\leavevmode
\epsfxsize =12 cm
\psfragscanon
\epsfbox{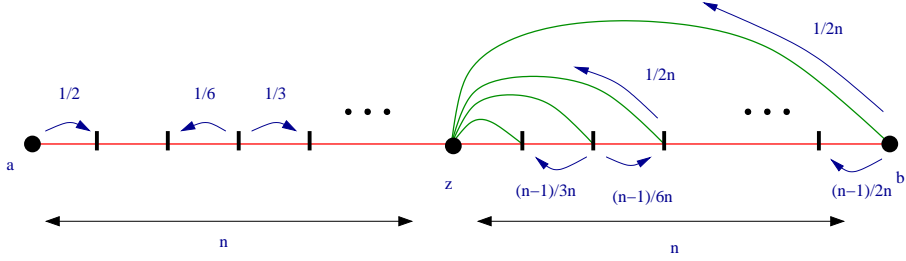}
\end{center}
\caption{\label{fig:ex1} 
A schematic representation of the transition rates for Example \ref{ex: 1}.
On the segments $A$ and $B$ the transition rates away from and towards the center of mass $z$ 
are equal respectively to $1/6$ and $1/3$ (on the $A$ side) and 
$(n-1)/6n$ and $(n-1)/3n$ (on the $B$ side). The rate for using a green-edge to land on $z$ is equal to $1/2n$. The rates for using green edges in the other direction has 
a more complicated expression prescribed by reversibility. These rates are described below despite the fact that they play no role in our analysis. 
}
\end{figure}

A simple calculation show that 
\begin{equation}
\label{eq: 4.3}
\begin{split}
w_n(z)&=1+\frac{3-2^{-(n-2)}}{2(n-1)},\\ 
\sum_{y\in \gO_n}w_n(y)&=4(1-2^{-n})+\frac{3-2^{-(n-2)}}{2(n-1)},
\end{split}\end{equation}
which implies  $\lim_{n\to \infty} \pi_n(z) =1/4$. 
The transition matrix obtained from $w_n$ is
\begin{itemize}
\item
$P_{n}(x,x)=1/2 $, for all $x \in \Omega_n $.
\item $P_n(a_{n},a_{n-1})=1/2$.
\item $2P_n(a_i,a_{i+1})=1/3=P_n(a_{i},a_{i-1})$, for all $1 \le i <n$.
\item $P_n(b_{i},z)=\frac{1}{2n} $, for $i\ge 2$.
\item $P_n(b_{i},b_{i-1})=\frac{1}{3}(1-\frac{1}{n})=2P_n(b_{i},b_{i+1}) $, for all $2\le i\le n-1$.
\item
$P_n(b_n,b_{n-1})=\frac{1}{2}-\frac{1}{2n} $.
\item
$P_n(b_{1},z)=\frac{1}{2n}+\frac{1}{3}(1-\frac{1}{n}) $ and  $P_n(b_{1},b_{2})=\frac{1}{6}(1-\frac{1}{n})
$.
\item  $P_n(z,b_1)=\frac{2n+1}{4(2n+1-2^{-(n-2)})}=\frac{1+o(1)}{4}$ and
$P_n(z,a_1)=\frac{n-1}{4n+2-2^{-(n-1)}}=\frac{1-o(1)}{4}$. 
\item $P_n(z,b_k)=\frac{3}{2^{\\ k+1}\left(2n+1-2^{-(n-2)}\right)}=\frac{3-o(1)}{ n 2^{k+2}}$, for $2 \le k \le  n-1$,\\
 and $P_n(z,b_n)=\frac{1}{2^{n}\left(2n+1-2^{-(n-2)}\right)}$.
\end{itemize}

Note that for this chain, condition \eqref{star} is easily verified using Theorem \ref{thm: Cheeger}. Since under $\Pr_{a_n}$, $T_z$ is concentrated around time $6n$, to prove total-variation cutoff around time $6n$
 for this sequence of chains (using Proposition \ref{prop: crittv}),
we only need to verify that $a_{n}$ is the initial state from which $T_z$ is (stochastically) the largest.
A crucial fact which shall assist us in this task is that for all $i \in [n]$ and all $t$
\begin{equation}
\label{eq: comparingBtoA}
\Pr_{b_{i}}[T_z >t] \le \Pr_{a_{i}}[T_z >t]. 
\end{equation}
The reason for this identity is the following: We couple  $X^A$ and $X^B$ starting from $a_i$ and $b_i$ (resp.) in the following manner:
with probability $1/2$
both stay put, with probability $(\frac{1}{2}-\frac{1}{2n})$ $X^A$ and $X^B$ make ``the same move" ($+/- 1$ (towards/away from $z$) with (conditional) probability $1/3$ and $2/3$, resp.~(unless the current position of the chain is either $a_n$ or $b_n$ in which case the move has to be $-1$) and
with probability $1/(2n)$, $X^B$ is sent directly to $z$ while $X^A$ moves towards/away from $z$ with probability $2/3$ and $1/3$ (unless it is located at $a_n$).
We do not need to specify how the coupling is defined after $X^B$ has hit $z$.

\medskip

A way to describe $X_t$ starting from $B$ before it hits $z$ is the following: at each step it is killed (hits $z$) 
with rate $1/(2n)$ and conditionally on not being killed, it performs ``the same" random walk as that on $A$ 
(in terms of the index of its current position) but with holding probability
$n/(2n-1)\ge 1/2$. 

Consequently,
\begin{equation}
\label{eq: cond2}
\begin{split}
& \max_{y \in \Omega_n }\Pr_{y}[T_z >t]=\max_{y \in A_n }\Pr_{y}[T_z
>t]=\Pr_{a_{n}}[T_z
>t], \\ &  \max_{y \in B_n }\Pr_{y}[T_z
>t] = \Pr_{b_{n}}[T_z
>t].
\end{split}
\end{equation}
Moreover, it follows from the above discussion that
\begin{equation}\label{eq: comparingBtoAbis}
 \Pr_{a_{n}}[T_z >t] \left(1-\frac{1}{2n}\right)^{t} \le \Pr_{b_{n}}[T_z >t] \le \min \left( \Pr_{a_{n}}[T_z >t],  \left(1-\frac{1}{2n}\right)^{t}\right).
\end{equation}

\medskip

We now turn to the task of verifying that there is no cutoff in separation. Note that conditions $(i)$-$(ii)$ of Proposition \ref{prop: critsep} hold by construction, condition $(v)$ holds by Remark \ref{rem: distance}, while condition $(iv)$ holds by \eqref{eq: septhroughz3}. Lastly, condition $(iii)$ of Proposition \ref{prop: critsep} 
follows form \eqref{eq: cond2}, and so Proposition \ref{prop: critsep} applies. Consequently,
\begin{equation}
\label{eq: applyingCor3.9}
\lim_{n\to \infty} \sup_{t\ge 0} |\dsepn(t)-\mathbb{P}[T_z^{a_n,b_n} > t]|=0.
\end{equation}
Set $m_n:=\lceil n^{2/3}\rceil$ (the exponent $2/3$ can be replaced by any number in $(1/2,1)$). It is standard to check that 
\begin{equation}
 \begin{split}
 & \quad \quad \quad \quad \lim_{n\to \infty} \Pr_{a_n}[|T_{z}-6n|>m_n]=0,\\
 &\lim_{n\to \infty} \sup_{t\in [0,6n-2m_n]}  | \Pr_{b_n} (T_z> t+m_n)-\Pr_{b_n} (T_z> t) |=0.
                \end{split}
\end{equation}
Hence it follows from \eqref{eq: applyingCor3.9} that for all $c\in (0,6)$
\begin{equation}
\label{eq: 5.5}
\lim_{n\to \infty}  \ | \dsepn(6n+\lfloor cn \rfloor) - \Pr_{b_{n}}[T_{z}> cn] \ | =0.
\end{equation}
This and \eqref{eq: comparingBtoAbis} yield that for any $0<\eps \le 1/4 $ 
\begin{equation}
\lim_{n\to \infty} \dsepn(\lfloor sn \rfloor)=\begin{cases} 1 &\text{ if } s\le 6, \\
                                e^{-(s-6)/2 } &\text{ if } s\in[6,12),\\
                                0 &\text{ if } s> 12.
                                \end{cases}
\end{equation}
Hence there is no separation cutoff. Moreover,
$$ \sup_{0<\eps<1/2} \liminf_{n \to \infty}\frac{\sepneps}{\sepn(1-\eps)}
= 2 .$$

We now describe a variant of the previous example which is a nearest neighbor lazy weighted random walk 
on a bounded degree graph with bounded transition probabilities.
\begin{example}
\label{ex: 2}
Let $\Omega_n=A \cup B \cup C \cup \{z \} $, where 
\begin{equation}
 \begin{split}
  A=A_{n}&:=\{a_{1},a_{2},\ldots,a_{2n}=a \},\\
  B=B_{n}&:=\{b_{1},b_{2},\ldots,b_{2n}=b\},\\
  C=C_{n}&:=\{c_{1},c_{2},\ldots,c_{n-1}\}.
 \end{split}
\end{equation}
For notational convenience we write $a_{0}:=z=:b_{0}=c_0$ and $c_{n}=b_{n}$. Consider
the following transition matrix
\begin{itemize}
\item
$P_n(x,x)=3/4$ for all $x \in \Omega_n \setminus C$ 
and $P_n(c_i,c_i)=1/2$ for all $i\in \{1,\dots,n-1\}$,
\item
$P_n(a_{2n},a_{2n-1})=1/4=P_n(b_{2n},b_{2n-1})$, 
\item  $2P_n(a_i,a_{i+1})=1/6=P_n(a_{i},a_{i-1})$, for all $1 \le i <2n$,
\item
$2P_n(b_i,b_{i+1})=1/6=P_n(b_{i},b_{i-1})$, for all $ i \in [2n-1] \setminus \{n\} $,
\item $2P_n(c_i,c_{i+1})=1/3=P_n(c_{i},c_{i-1})$, for all $1 \le i \le n-1$,
\item $P_n(b_{n},b_{n+1})=P_n(b_{n},c_{n-1})=P_n(b_{n},b_{n+1})=1/12$,
\item $P_n(z,a_1)=P_n(z,b_1)=P_n(z,c_1)=1/12$.
\end{itemize}
\end{example}

\begin{figure}[h]
\begin{center}
\leavevmode
\epsfxsize =12 cm
\psfragscanon
\epsfbox{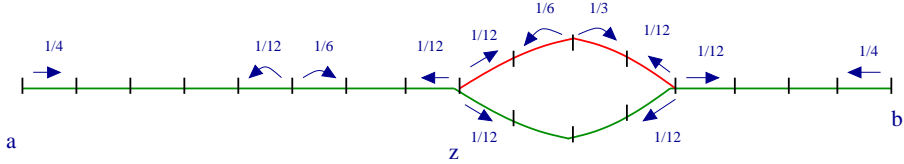}
\end{center}
\caption{\label{fig:ex2} 
A schematic representation of the transition rates for Example \ref{ex: 2} for $n=4$. When at a state of degree two or three (other than $z$), conditioned on making a non-lazy step, the chain moves away from (resp.~towards) $z$ with conditional probability $1/3$ (resp.~$2/3$).
For vertices of degree $2$: along the green edges, rates away from and towards the center of mass $z$ 
are equal respectively to $1/12$ and $1/6$ and  along the red edges 
they are equal to  $1/6$ and $1/3$, respectively.
The transitions away from vertices of degree $1$ and $3$ are given on the figure.}
\end{figure}

 States $a_{2n}$, $b_{2n}$ and $z$ play here the same respective roles as $a_{n}$, $b_{n}$ and $z$  in the previous example.
A simple calculation (similar to (\ref{eq: 4.3})) yields that 
$$\lim_{n\to \infty} \pin (z)=2/7.$$ We argue that for all $t \ge 0$ and $i \in [2n]$
\begin{equation}
\label{eq: cond2'}
\max\left(\Pr_{c_{i}}[T_z >t],\Pr_{b_{i}}[T_z >t]\right) \le \Pr_{a_i}[T_z>t] \le \Pr_{a_{2n}}[T_z>t]
\end{equation}
In particular,
\begin{equation}
\label{eq: cond3}
\forall t\ge 0, \quad \max_{x \in \Omega_n} \Pr_{x}[T_z >t]=\Pr_{a_{2n}}[T_z>t]. 
\end{equation}
Since the hitting time of $z$ under $\Pr_{a_{2n}}$ is concentrated around
time $t=24n$, by Proposition \ref{prop: crittv} the sequence exhibits
total-variation cutoff around time $24n$.

The last inequality in  \eqref{eq: cond2'} is trivial.
For the first one we consider the case where $P_n$ is replaced by $P'_n$ 
which satisfies $2P'_n(c_i,c_{i+1})=1/3=P'_n(c_{i},c_{i-1})$ and $P_n'(c_i,c_i)=1/2$ for $1\le i\le n-1$ and $P'(x,y)=P(x,y)$ 
elsewhere.
As adding extra laziness increases stochastically the hitting time $T_z$ (as in Remark \ref{rem: distance} consider the law of $\gamma$, the non-lazy path performed by the chain by time $T_z$; Clearly it is invariant under this transformation, while the conditional law of $T_z$, given $\gamma$, can only increase, stochastically),
\begin{equation}
  \Pr_{b_{i}}[T_z >t] \le \Pr'_{b_{i}}[T_z >t] = \Pr_{a_i}[T_z>t],
\end{equation}
(where $\Pr'$ denotes the distribution of the modified chain with the increased holding probability on $C_n$) and the same holds when $b_i$ is replaced by $c_i$.

\medskip

 To prove that $b_{2n}$ is the vertex from which the hitting time of $z$ is the largest, 
we need to prove the following two inequalities valid for $i \in \{1,\dots,n\}$
\begin{equation}
 \Pr_{c_i}[T_z>t] \le  \Pr_{b_i}[T_z>t] \quad \text{ and } \quad  \Pr_{b_i}[T_z>t] \le \Pr_{b_{i+n}}[T_z>t].
\end{equation}
Both can be proved by coupling arguments. For the first one, we can couple the non-lazy path of the chains starting from $b_i$ and $c_i$ 
until they reach either $b_n$ or $z$ (the second being at position $c_j$ when the first is at position $b_j$), and then in the case 
they reach $b_n=c_n$ let them evolve together until they reach $z$. The larger laziness on the path starting from $c_i$ until the merging time,
implies stochastic domination.
For the second inequality, the case $i=n$ follows from fact that starting from $b_{2n}$ the chain has to go through $b_n$ before reaching $z$.
For $i<n$, we can couple the chain starting from $b_i$ and $b_{i+n}$ until the pair of chains reaches either $(b_n,b_{2n})$ or $(z,b_n)$ (the second chain 
being at position $b_{j+n}$ when the first is at position $b_j$), and conclude using the case $i=n$.

As in the previous example, we can apply Proposition \ref{prop: critsep}. 
The reason why separation cutoff does not occur is that when starting from $b_{2n}$,
the hitting time $T_z$ is not concentrated. Indeed it is concentrated around $18n$ under the conditioned probability measure 
 $\Pr_{b_{2n}}[ \cdot  \mid X_{T_{z}-1}=c_1 ]$, while it is concentrated around $24n$ under  $\Pr_{b_{2n}}[\cdot
\mid X_{T_{z}-1}=b_1]$. As by symmetry 
$$ \Pr_{b_{2n}}\left[ X_{T_{z}-1}=c_1 \right]=\Pr_{b_{2n}}\left[ X_{T_{z}-1}=b_1 \right]=\frac{1}{2},$$
this yields 

\begin{lemma}
 We have 
 \begin{equation}
\lim_{n\to \infty}  \Pr_{b_{2n}}[T_{z}\ge sn ]=\begin{cases} 1 \text{ if } s< 18,\\
  1/2   \text{ if } s\in (18,24),\\
   0   \text{ if } s>24.
                           \end{cases}
 \end{equation}
 \end{lemma}
While this result is rather elementary (we use some surgery to compare $T_z$ with a sum of independent variables, 
and then the law of large number for this sequence),
the proof in full detail is long to expose (c.f.~\cite[Example 8.1]{cf:Basu}) and we choose to leave it as an exercise.
Applying Proposition \ref{prop: critsep} for an adequate choice of sets and states (here $(a_{2n},b_{2n},A_n,B_n \cup C_n)$ plays the role of $(a_n,b_n,A_n,B_n)$ 
from Proposition \ref{prop: critsep}) yields
 \begin{equation}
 \lim_{n\to \infty} \dsepn(sn) =\begin{cases} 1 \text{ if } s< 42,\\
  1/2   \text{ if } s\in (42,48),\\
   0   \text{ if } s>48.
                           \end{cases}
 \end{equation}
 In particular, there is no cutoff in separation.

\section{Separation cutoff without total variation cutoff example}
\label{sec: Sepcutoff}
In the following example the analysis of the sharp transition of $\dsepn (t)$ is reduced to the analysis of the behavior of sum of i.i.d.~random variables in the large deviation regime. The analysis below is too coarse for the purpose of determining the width of the cutoff window. We later present a refined analysis for Example \ref{ex: 5} (which is the bounded degree un-weighted version of Example \ref{ex: 3}) in \S~\ref{sec: Rem 1.6}, which shows that in fact $\sepn (\gep)-\sepn (1-\gep) \le C |\log \gep |$, for some absolute constant $C>0$. The  analysis in \S~\ref{sec: Rem 1.6} is built upon the analysis of Example \ref{ex: 3} below, as it relies (in a non-quantitative manner) on the fact that certain large deviation estimates hold uniformly over compact sets (the identity of the large deviation rate function is not important for the analysis in \S~\ref{sec: Rem 1.6}).\begin{example}
\label{ex: 3}
Let $M \ge 10$ be a fixed integer whose exact value shall be determined later. Consider the state space $\Omega_{n}=A \cup B \cup \{z\} \cup C \cup D \cup \{z' \} $, where $A=A_{n}:=\{a_{Mn}=a,a_{Mn-1},\ldots,a_1
\}$, $B=B_{n}:=\{b_{Mn}=b,b_{Mn-1},\ldots,b_{1}
\}$, $C=C_{n}:=\{c_{1},c_{2},\ldots,c_{n-1}\}$ and $D=D_{n}:=\{d_{1},d_{2},\ldots,d_{n-1}\}$. We use the following notational convention:
$a_{0}=b_{0}:=z'=:c_{n}=d_{n}$ and $d_{0}:=z=:c_{0}$. Consider
the following transition matrix
\begin{itemize}
\item
$P_{n}(i,i)=\begin{cases}3/4 & i \in C, \\
1/2 & \text{otherwise.}
\end{cases}$
\item
$P_{n}(a_{Mn},a_{Mn-1})=1/2=P_{n}(b_{Mn},b_{Mn-1})$.
\item
$P_{n}(z,c_{1})=1/4=P_{n}(z,d_{1})$.
\item
$P_{n}(z',c_{n})=P_{n}(z',d_{n})=1/6=2P_{n}(z',a_{1})=2P_{n}(z',b_{1})$.
\item
$P_n(a_i,a_{i-1})=P_{n}(b_i,b_{i-1})=P_n(d_j,d_{j-1})=2P_{n}(c_{j},c_{j-1})=1/3.$\\
$P_n(a_i,a_{i+1})=P_{n}(b_i,b_{i+1})=P_n(d_j,d_{j+1})=2P_{n}(c_{j},c_{j+1})=1/6$,\\ 
for all $i \in [Mn-1] $ and $j \in [n]$.
\end{itemize}
\end{example}

\begin{figure}[h]
\begin{center}
\leavevmode
\epsfxsize =12 cm
\psfragscanon
\epsfbox{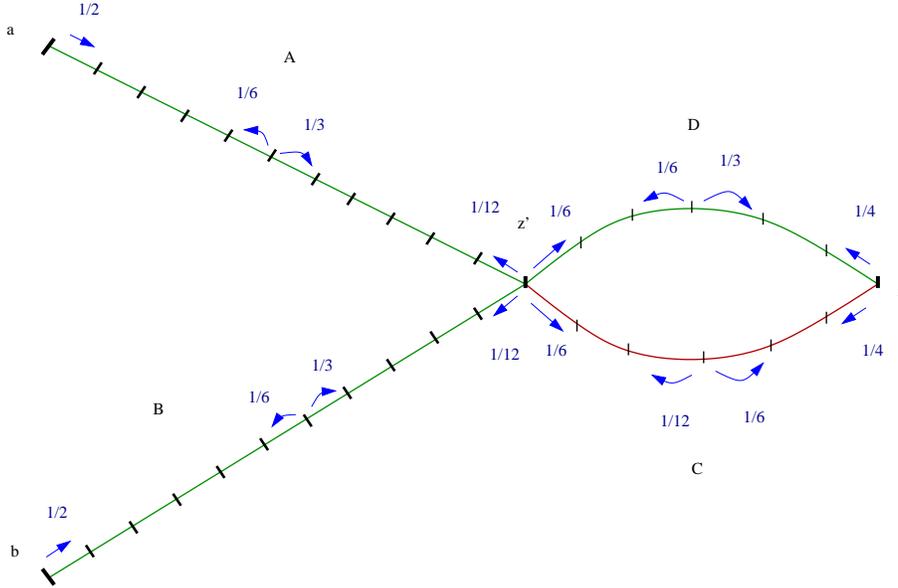}
\end{center}
\caption{\label{fig:ex3} 
A schematic representation of the transition rates for Example \ref{ex: 3}. When at a state of degree two or four (other than $z$), conditioned on making a non-lazy step, the chain moves away from (resp.~towards) $z$ with conditional probability $1/3$ (resp.~$2/3$).
The transition rates away from and towards the center of mass $z$, from degree two states, are equal respectively to $1/6$ and $1/3$, except on the segment $C$, due to increased holding probability. 
The transition rates away from the rest of the states are specified in the figure.}
\end{figure}

This chain is a modification of Aldous' example (which was discussed in \S~\ref{sec: overview}). 
The difference lies in the introduction of an additional branch $B$ to the graph. This branch has no effect on the 
total-variation profile of the convergence to equilibrium, but crucially modifies the separation profile, as $P_{n}^t(a,b)/\pi_n(a)$ 
(recall $a:=a_{nM}$ and $b:=b_{nM} $) is the quantity that takes 
the longest time to reach equilibrium (i.e.~up to negligible correction $(x,y)=(a,b)$ maximizes $1-P_{n}^t(x,y)/\pi_{n}(y)$ for all relevant $t$).

\medskip

A standard calculation yields that
\begin{equation}
\label{eq: 5.1}
\lim_{n \to \infty} \pin(z)=2/11 \text{ and }  \lim_{n \to \infty} 2^n\pin(z')=6/11.
\end{equation}
By symmetry, the law of  $T_{z}$ starting, resp., from $a_i$ and $b_i$ is identical for all $i$ and by the Markov property, it is stochastically increasing in $i$ (for $i>j$, to reach $z$ from $a_{i} $ (resp.~$b_{i}$) the chain must first hit $a_j$ (resp.~$b_j$)). 
Only minor efforts are necessary to prove rigorously that  $a$ and $b$ are the points in $A \cup B \cup D $ 
for which the hitting time $T_{z}$ is stochastically the largest (the coupling arguments are similar to the one developed in the previous section),
while for any choice of $M> 1$,  
$$\limsup_{n \to \infty} \sup_{t} (\max_{c \in C}\Pr_c[T_z>t]- \Pr_a[T_z>t]) \le 0.$$ Due to the different holding probabilities along the two branches, $C,D,$
the distribution of $T_{z}$ under $\Pr_{a}$ is not concentrated around its mean.
Thus, by Proposition \ref{prop: crittv}, there is no total-variation cutoff, and the total-variation asymptotic profile is given by
\begin{equation}
\label{eq: tvbnds}
 \lim_{n\to \infty}d^{(n)}(6sn):=
\begin{cases}
0 \quad &\text{ if } s< M+1,\\
1/2 \quad &\text{ if } s\in (M+1,M+2),\\
1  \quad &\text{ if } s>M+2.
\end{cases}
\end{equation}
To show that there is separation cutoff, it suffices to prove that
\begin{equation}
\label{eq: reductionex3}
\liminf_{n\to \infty}\inf_t \min_{x,y \in \Omega_n } P_{n}^t(x,y)/\pi_{n}(y)- \min \left( 1, P_{n}^t(a,b)/\pi_{n}(b) \right)=0,
\end{equation}
and to show that
 $\min(1,P_{n}^t(a,b)/\pi_{n}(b))$
displays an abrupt transition.
Let us start with the second point. According to Lemma \ref{lem: pathsdecomposition} (first inequality of \eqref{eq: septhroughz1}), we have
\begin{equation}\label{eq:lineq}
\bbP[T^{a,b}_{z'} = t]/\pi_n(z') \le P_{n}^t(a,b)/\pi_{n}(b) \le \bbP[T^{a,b}_{z'} \le t]/\pi_n(z')
\end{equation}
By definition $T^{a,b}_{z'}$ is the sum of two independent hitting times of a biased random walk on a segment of length $Mn$ (from one end-point towards the one towards which there is a bias).
We make some efforts to compute the large deviation behavior of this sum.

\begin{lemma}\label{largedev}
Consider a lazy random walk $(Z_t)_{t\ge 0}$ on $\bbZ_+$ with rates $p(x,x+1)=1/3$, $p(x+1,x)=1/6$, $x\in \bbZ_+$.
Let $\cT_N$ be the first hitting time of $N$.
We have
 \begin{equation}\label{oops}\begin{split}
  \lim_{N\to \infty} \frac {1}{N}\log \bbP[\cT_{N}= \lfloor sN \rfloor ]=  \lim_{N\to \infty} \frac {1}{N }\log \bbP[\cT_N\le sN]&=-\Psi(s), \text{ for } s\in[1,6]\\
  \lim_{N\to \infty} \frac {1}{N}\log \bbP[\cT_{N}= \lfloor sN \rfloor ]=  \lim_{N\to \infty} \frac {1}{N }\log \bbP[\cT_N\ge sN]&=-\Psi(s), \text{ for } s\ge 6,
\end{split}
  \end{equation}
 where $\Psi$ is the following Legendre transform 
 \begin{equation}
  \Psi(s):=\sup_{\gl\in (-\infty,\infty)} \left[ \gl s- \log \tf(\gl) \right],
 \end{equation}
 where 
 \begin{equation*}
  \tf(\gl):=\begin{cases}
             \infty   & \text{if } \gl > \log (6/(3+2\sqrt{2})),\\
            \frac{6e^{-\gl}-3}{2}- \frac{\sqrt{(6e^{-\gl}-3)^2-8}}{2} \quad &\text{ if } \gl \le \log (6/(3+2\sqrt{2})).  
            \end{cases}
 \end{equation*}
Moreover, $\Psi(6)=0=\Psi'(6)$ and the second derivative $\Psi''(6)$ is positive.
 \end{lemma}
 
 \begin{proof}
  Let $X'$ be the random walk with the same rates on $\bbZ$, and $\cT'_N$ be the first hitting time of $N$ for this walk.
By the Markov property $\cT'_N$ is the sum of $N$ IID copies of $\cT'_1$ and hence we can use Cram\'er's Theorem (see e.g.\ \cite[Chapter 2]{cf:DZ}) 
to obtain the large deviation for $\cT'_N$ below its mean.
If one decomposes according to the value of $X'_1$ we notice that the Laplace transform  $\tf(\gl):=\bE[e^{\gl\cT'_1}]$ satisfies
\begin{equation}\label{degrau}
 \tf(\gl)=e^{\lambda}\left(\frac{1}{3}+\frac{1}{6}\tf(\gl)^2+\frac{1}{2} \tf(\gl)\right).
\end{equation}
and we deduce the right value for $\tf(\gl)$ from this relation  (the fact that $\tf(0)=1$ and continuity of $\tf$ indicates which root to choose in \eqref{degrau}). Note that the derivative of $\log \tf(\gl)$ at zero is equal to $6$
which implies that $\Psi(6)=0$ (Alternatively, $\mathbb{E}[\cT'_1]=6$, hence by Cram\'er's Theorem it must be the case that $\Psi(6)=0$). As $\Psi$ is non-negative (since $\log \tf(0)=0$), it must be the case that it attains a global minimum at $6$, which implies that $\Psi'(6)=0$ and $\Psi''(6)>0$.
\medskip

Now, note that $\cT_i-\cT_{i-1}$ are independent variables, which are dominated by $\cT'_1$ and who converge (when $i$ tends to infinity) to $\cT'_1$ in law.
In particular, by dominated convergence (and Cesaro's Theorem) we have that for any $\gl \in   (-\infty,\log \frac{6}{3+2\sqrt{2}}]$, 
\begin{equation}
 \lim_{N\to \infty} \frac{1}{N}\log \bE[e^{\gl \cT_N}]=\tf(\gl).
\end{equation}
and thus in that case the result follows from G\"ardner Ellis Theorem \cite{cf:DZ}.
Finally, the local large deviation estimate (the result on $\bbP[\cT_{N}= \lfloor sN \rfloor]$) can be deduced from the large deviation principle
using the fact that due to laziness
\begin{equation}\label{lazy}
\frac{\bP[\cT_N=t+1]}{\bP[\cT_N=t]}\ge \frac{1}{2}.
\end{equation}
We leave it as an exercise.
Note moreover that the convergence \eqref{oops} holds uniformly on $s\in K$ for any compact $K$ (it can be deduced e.g.\ from \eqref{lazy}).
\end{proof}

A consequence of \eqref{eq:lineq} and the previous lemma in conjunction with \eqref{eq: 5.1} and Lemma \ref{lem: pathsdecomposition} 
is that 
if $s_M$ is given by $2Ms^*$, where $s^*$ is the unique solution in $(0,6)$ of 
\begin{equation}\label{eq:theeq}
 2M\Psi\left(s \right)=\log 2,
\end{equation}
then
\begin{equation}
 \lim_{n\to \infty} \frac{P_{n}^{\lfloor sn \rfloor}(a,b)}{\pi_{n}(b)}=\begin{cases} 0, \text{ if } s<s_M,\\
                                                         \infty, \text{ if } s\in (s_M,12M].
                                                        \end{cases}
 \end{equation}
An order $2$ Taylor expansion  of \eqref{eq:theeq} around $s=6$ readily shows that $6-s^*=\Theta(1/\sqrt{M})$ (i.e.~$12M-s_M=\Theta(\sqrt{M})$) for large $M$. 
In particular, $s_M\ge 11 M$  for $M$ sufficiently large.
What is left to do in order to prove separation cutoff is to check that for any $s\in (s_M,12M]$ (in fact, by monotonicity it suffices to consider only $s$ arbitrarily close to $s_M$)
we have 
$$\liminf_{n\to \infty}\min_{x,y} P_{n}^{\lfloor sn \rfloor}(x,y)/\pi_{n}(y))\ge  1. $$
In what follows we let $s \in (s_M,12M] $ be fixed.
\medskip

We first use Lemma \ref{lem: pathsdecomposition} to reduce to the case of $x=a_i$, $y=b_j$, $i,j\ge Mn/2$.
Set $E:=\{a_i:i \ge \frac{Mn}{2}  \} \cup \{b_i:i \ge \frac{Mn}{2}  \}$. 
By \eqref{eq: septhroughz1} for any $x \in \Omega_n$ and $y \in \Omega_n \setminus E$ 
we have 
\begin{equation}
\label{eq: onlyfarmatters}
\frac{P_n^{\lceil s n \rceil} (x,y)}{\pin(y)} \ge \mathbb{P}\left[T_{z}^{x,y}\le 11 Mn\right],
\end{equation}
and it is a simple exercise to show that (when $M$ is sufficiently large) 
\begin{equation}
 \lim_{n\to \infty}\min_{(x,y)\in \gO_n \times (\Omega_n \setminus E)}\mathbb{P}\left[T_{z}^{x,y}\le 11 Mn\right]=1.
\end{equation}

Finally, to treat the case $x=a_i$, $y=b_j$ (the cases $(a_i,a_j)$ or $(b_i,b_j)$ are treated in the same manner), $i,j\ge Mn/2$ , we use again
Lemma \ref{lem: pathsdecomposition} which asserts that
\begin{equation}\label{eq:oneortheother}
\frac{P_n^{\lfloor s n \rfloor} (a_i,b_j)}{\pi_n(b_j)}\ge \max \left(  \frac{\bbP[T^{a_i,b_j}_{z'}= \lfloor sn \rfloor ]}{\pi_n(z')},\bbP[T^{a_i,b_j}_{z'}\le \lfloor sn \rfloor ]  \right)
\end{equation}
Note that  
$T_{z'}^{a_i,b_j}$  is (cf.~the proof of Lemma \ref{largedev}) a sum of  $i+j$ independent random variables (not identically distributed)
and that
\begin{equation}
\lim_{n\to \infty} \sup_{i,j\ge Mn/2} \left | \frac{1}{i+j} \log \bbE[e^{\gl T^{a_i,b_j}_{z'}}]-\log \tf(\gl)\right|=0.
\end{equation}
One deduces from G\"ardner Ellis Theorem \cite{cf:DZ} and the following consequence of laziness
\begin{equation}
\frac{\bbE[T^{a_i,b_j}_{z'}=t+1]}{\bbE[T^{a_i,b_j}_{z'}=t]}\ge 1/2,
\end{equation}
that for any $u\in (1,\infty)$,
\begin{equation}\label{eq:largedev}
 \lim_{n\to \infty} \sup_{i,j\ge Mn/2}  \left|- \frac{1}{i+j} \log \bbP[T^{a_i,b_j}_{z'}=\lfloor (i+j) u \rfloor] -\Psi(u) \right|=0,
\end{equation}
and the convergence holds uniformly on compact sets.
Now let us fix $\eta$ which satisfies 
$$\Psi\left( \frac{6s}{s+6\eta} \right) \le \frac{\log 2}{4M},$$
Using \eqref{eq:largedev}, there exists $\delta$ such that for all $n$ sufficiently large
,  for all $i,j\ge Mn/2$,
\begin{equation}\begin{split}\label{eq:dabound}
 \left( i+j\le \frac{sn}{6}+\eta n \right) \quad &\Rightarrow \bbP[T^{a_i,b_j}_{z'}\le \lfloor sn \rfloor ]\ge 1-e^{-\delta n},\\
\left( i+j\ge \frac{sn}{6}+\eta n \right) \quad  &\Rightarrow \log \bbP[T^{a_i,b_j}_{z'}= \lfloor sn \rfloor ]\ge -(i+j)\left[\Psi(\frac{sn}{i+j}) + \delta\right],
\end{split}\end{equation}
 Note that the l.h.s.\ in the second line satisfies
\begin{equation}\label{eq:control}
 (i+j)\left[\Psi(\frac{sn}{i+j})  +  \delta\right]\le 2Mn\left( \max\left[ \Psi\left(\frac{s}{2M}\right), \Psi\left( \frac{6s}{s+6\eta}\right)\right]+\delta\right).
\end{equation}
 As $2M \Psi\left(\frac{s}{2M}\right)<\log 2$ (since $s \in (s_M,12M]$) and $\delta$ can be chosen arbitrarily small, 
 \eqref{eq:dabound} (second line) and \eqref{eq:control} imply that for sufficiently large $n$, for any $i,j$ satisfying 
 $i+j\ge \frac{sn}{6}+\eta n$, we have
\begin{equation}
\bbP[T^{a_i,b_j}_{z'}= \lfloor sn \rfloor ]\ge 2^{-n(1-\delta)}.
\end{equation}
Combining this with \eqref{eq:dabound} (first line) and \eqref{eq:oneortheother} we can conclude that 
\begin{equation}
 \min_{i,j\ge Mn/2} \frac{P_n^{\lfloor s n \rfloor} (a_i,b_j)}{\pi_n(b_j)}\ge 1-e^{-\delta n}.
\end{equation}

%

\subsection{Concerning Remark \ref{rem:ration}}

Note that by performing a minor modification in the above construction we can bring the pre-cutoff ratio for total-variation to the largest possible value: $2$.
A way to achieve this is to make one of the branches linking $z'$ to $z$ much faster than the other 
(instead of only twice faster as in Example 3, we want the ratio of speeds to tend to infinity).

\medskip

What we can do is to make these branches of length $\lceil \sqrt{n}\ \rceil $ while $A$ and $B$ are of length $n$. Furthermore, we choose 
the speed on one branch to be $1/6$ while that one the other being 
$1/(6 \sqrt{n})$ by increasing the holding probability on this branch (see Figure \ref{fig:ratiotwo}). Using similar reasoning as in the analysis of Example \ref{ex: 3} one can show that for this construction there is separation cutoff around time $12n$ (note that here $-\log \pi_n(z')=\Theta(\sqrt{n}) $, which by \eqref{eq:lineq} implies that for $t_n:=\lceil(12-\eps)n\rceil$, $P_{n}^{t_{n}}(a,b)/\pi_{n}(b) \le \bbP[T^{a,b}_{z'} \le t_{n} ]/\pi_n(z')=o(1)$, for every $\eps >0$). 

We can also find a similar example with transition rates bounded uniformly from zero 
by considering two branches of 
different lengthes, but in that case the analysis turns out to be more intricate.

\begin{figure}[h]
\begin{center}
\leavevmode
\epsfxsize =12 cm
\psfragscanon
\psfrag{a}{\tiny $a$}
\psfrag{b}{\tiny $b$}
\psfrag{z}{\tiny $z$}
\psfrag{z'}{\tiny $z'$}
\psfrag{n}{\tiny $n$}
\psfrag{sqrtn}{\tiny $\sqrt{n}$}
\psfrag{1/3}{\tiny $1/3$}
\psfrag{1/6}{\tiny $1/6$}
\psfrag{1/6\sqrtn}{\tiny $1/(3\sqrt{n})$}
\psfrag{1/3\sqrtn}{\tiny $1/(6\sqrt{n})$}
\psfrag{1/12}{\tiny $1/12$}
\psfrag{1/8}{\tiny $1/8$}
\psfrag{1/4}{\tiny $1/4$}
\psfrag{1/2}{\tiny $1/2$}
\epsfbox{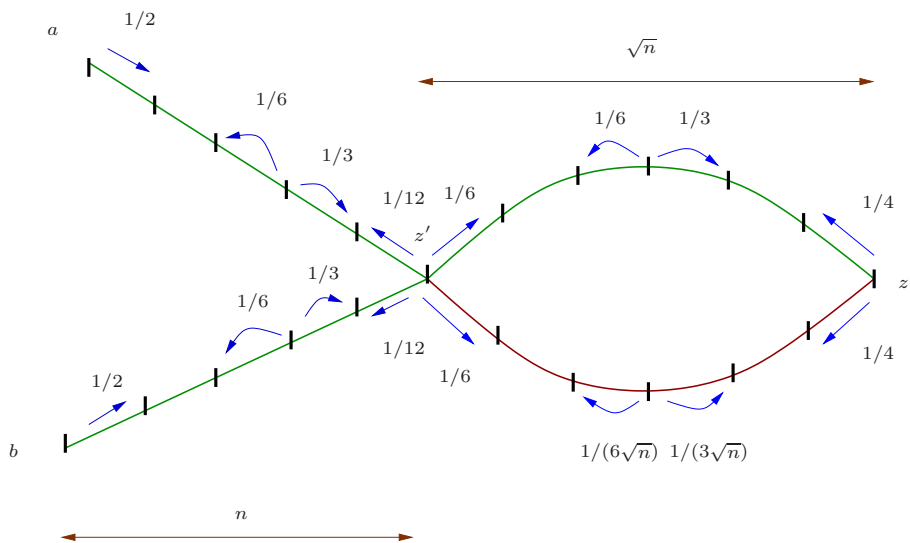}
\end{center}
\caption{\label{fig:ratiotwo} 
A modification of the graph size and of the holding probability along the slow branch as shown above yields a counter-example with  cutoff in separation and the maximal possible pre-cutoff ratio $2$
in total-variation.
}
\end{figure}

%
%
%

\section{Transforming Examples \ref{ex: 2} and \ref{ex: 3} into lazy simple random walk on bounded degree expander graphs}
\label{sec: LSRW}

\subsection{General comments and preliminaries}
\label{sec: generalizations} 
In this section we transform Examples \ref{ex: 2} and \ref{ex: 3} into lazy SRWs on a sequence of bounded degree expander graphs.
For this kind of walk, the equilibrium measure is $\pi(v)=\frac{\deg v}{\sum_u \deg u}$, and thus
no particular vertex can have the role of the ``center of mass" as in the previous examples.
Let us rewrite the definition of the Cheeger constant in this context.

\begin{definition}
Let $G=(V,E)$ be a finite connected graph. For every $S \subset V $ denote $C_{S}:=\sum_{v \in S}\deg v$. For any $S \subset V $ we define its {\em \textbf{edge boundary}} $\partial_{E}S $  to be the collection of edges having one vertex in $S$ and the other in $V \setminus S$. The Cheeger constant of lazy simple random walk on $G$ is defined as
$$\mathrm{ch}_{\mathrm{Lazy}}(G):=\min_{S:\pi(S) \le 1/2}|\partial_{E}S|/2C_S, $$
which coincides with Definition \ref{def: Cheeger} (see e.g.~\cite[Remark 7.2]{cf:LPW}).
We say that $G$ is a $c$-lazy expander if $\mathrm{ch}_{\mathrm{Lazy}}(G)>c $. We say that a sequence of finite graphs $(G_n)_{n \ge 1}$ is a family of $c$-lazy expanders if $\inf_n \mathrm{ch}_{\mathrm{Lazy}}(G_{n})>c
$.
\end{definition}

\medskip

In our new context, the center of mass is rather a set which contains a positive fraction of the vertices. We shall relate the mixing-time of the chain to the hitting time of this set. \textit{Mutatis mutandis}, 
the results of Section \ref{sec: Pre} and in particular Lemma \ref{lem: pathsdecomposition} can be adapted to this new context, but only if the 
set and the starting point satisfy a special relation:

\begin{definition}[Balanced sets]
\label{def: hitting}
For any $Z \subset \Omega$ we denote the {\em \textbf{hitting time}} of $Z$ by $T_{Z}:=\inf
\{t:X_t \in Z \} $.
\begin{itemize}
 \item  We say that $Z$ is \emph{\textbf{balanced}}
seen from $x\in \gO$ if for all $t$ such that $\Pr_x[T_{Z}=t]>0$, 
\begin{equation}
 \forall z\in Z, \quad \Pr_{x}[X_t=z
\mid T_{Z}=t]=\pi_{Z}(z),
\end{equation}
 where $\pi_{Z}(\cdot)=\frac{1_{\cdot \in Z}\pi (\cdot)}{\pi(Z)}$ is $\pi$ conditioned on the set $Z$.
 \item  We say that $Z$ is balanced seen from the set $A$ if it is balanced seen from $x$ for all $x\in A$.
 \item We define $T_{Z}^{x,y}$ to be a random
variable distributed like the sum
of two independent realizations of $T_Z$, once under $\mathrm{P}_x$ and once
under $\mathrm{P}_y$. That is, for all $t \ge 0$,
\begin{equation}
\mathbb{P}[T_{Z}^{x,y}=t]:=\sum_{k=0}^t\mathrm{P}_x[T_Z=k]\mathrm{P}_y[T_Z=t-k].
\end{equation}

\end{itemize}

\end{definition}

Note that sets are  not likely to be balanced by ``pure luck" and 
we will be careful to introduce a sufficient amount of symmetry when constructing our graphs, so that our center of mass will be balanced seen from many starting points.
However, this property cannot be satisfied for all starting points and we will have to 
deal with the remaining initial vertices separately (and show that they are irrelevant for determining the worst-case total-variation and separation distances), by using a crude $\ell_2$ estimate (Lemma \ref{lem: CS}).


\begin{lemma}
\label{lem: pathsdecompositions}
Let  $(\Omega,P,\pi)$ be a finite irreducible lazy reversible Markov
chain and consider $x,y
\in \Omega$, and $Z\in \gO$ which is balanced seen from both $x$ and $y$ 
\begin{itemize}
\item[(i)] For all $t \ge 0 $
we have 
\begin{equation}
\label{eq: septhroughz11}
P^t(x,y)/\pi(y) \ge \sum_{k \le t }\mathbb{P}[T_{Z}^{x,y}=k]\Pr^{t-k}_{\pi_Z}(Z)/\pi(Z)
\ge   \mathbb{P}[T_{Z}^{x,y} \le t]. 
\end{equation}
\item[(ii)]
If
$\mathrm{P}_{x}[T_Z<T_y]=1$ (i.e.~if every path from $x$ to $y$ goes through the set $Z$) then for all $t \ge 0$
we have that
\begin{equation}
\label{eq: septhroughz21}
\begin{split}
& P^t(x,y)/\pi(y) 
=\sum_{k \le t }\mathbb{P}[T_{Z}^{x,y}=k]\Pr^{t-k}_{\pi_Z}(Z)/\pi(Z)
 \\ & \le \mathbb{P}[T_{Z}^{x,y} \le t]+\frac{1}{2}t_{\mathrm{rel}} \max_{k\in \bbN} \bbP [T_{Z}^{x,y} =k ]\sqrt{(1-\pi(Z))/ \pi(Z)}.
\end{split}
\end{equation}
\end{itemize}
\end{lemma}
We use this result directly but also to prove the following key propositions whose aim is to replace Propositions 
\ref{prop: crittv} and \ref{prop: critsep}.

\begin{proposition}
\label{prop: hitcutoff}
Let $(\Omega_n,P_n,\pi_n)$ be a sequence of lazy reversible irreducible
finite chains which satisfies the product condition. Assume that for each $n$ there exist sequences of sets and vertices  $I_n,Z_n \subset
\Omega_n$, $a=a(n) \in \Omega_n$ which satisfy
\begin{itemize}
\item[(i)]
$\inf_{n } \pi_n(Z_n) > 0$.
\item[(ii)] $Z_n$ is balanced seen from $I_n$ for all $n$.
\item[(iii)] $\limsup_{n \to \infty} \sup_{t\ge 0}  \max_{i \in I_n} \Pr_i[T_{Z_{n}}>t]-\Pr_a[T_{Z_{n}}>t]\le 0$.
\item[(iv)]  $ \limsup_{n \to \infty} \sup_{t\ge 0} \max_{x
\in \Omega_n \setminus I_n}\|\Pr_x^t-\pi \|_{\TV} - \Pr_a[T_{Z_{n}}>t] \le 0.$ 
\end{itemize}
Let $\tau_{n}(p):=\inf \{ t:\Pr_a[T_{Z_{n}}>t] \le p \}$.
Then
\begin{equation}
\label{eq: hitcutoff01}
 \limsup_{n \to \infty}\frac{\mixn(\eps')}{\tau_n(\eps)} \le 1 \text{ and } \liminf_{n \to \infty}\frac{\mixn(\eps)}{\tau_n(\eps')} \ge 1, \text{ for all }0<\eps<\gep' <1 .
 \end{equation}
In particular, total-variation cutoff occurs if and only if
\begin{equation}
\label{eq: hitcutoff1}
  \lim_{n \to \infty} \frac{\tau_{n}(\epsilon)}{\tau_{n}(1-\epsilon)} =1,
\text{ for every } 0<\epsilon<1.
\end{equation}
\end{proposition}

\begin{proposition}
\label{cor: sepcutoffcriterion}
Let $(\Omega_n,P_n,\pi_n)$ be a sequence of lazy reversible irreducible
finite Markov chains which satisfies \eqref{star}. Assume that there exist sequences of sets and vertices,  $A_n,B_n,Z_n \subset
\Omega_n$, $a_n\in A_n$, $b_n \in B_n$, which satisfy
\begin{itemize}
\item[(i)]
$\inf_n \pi_n(Z_n) >0$.
\item[(ii)] $\lim_{n\to \infty} \max_k \Pr_{a_n}[T_{Z_n}=k] =0$.
\item[(iii)] $Z_{n}$ is balanced seen from $I_n:= A_n \cup B_n $.
\item[(iv)] $\Pr_x[T_{Z_n}<T_{B_n}]=1$ for all $x \in A_n$.
 \item[(v)] 
 \begin{equation}\begin{split}
\limsup_{n\to \infty} \sup_{t\ge 0} \max_{y \in B_{n}} \Pr_{y}[T_{Z_{n}}>t]- \Pr_{b_n}[T_{Z_{n}}>t]&= 0, \\
\limsup_{n\to \infty} \sup_{t\ge 0} \max_{y \in I_n} \Pr_{y}[T_{Z_{n}}>t]- \Pr_{a_n}[T_{Z_{n}}>t]&=0.
\end{split}
\end{equation}
\item[(vi)]
$$
\liminf_{n\to \infty} \inf_{t\ge 0} \min_{(x,y) \in A_n^2 \cup ( \gO_{n}^2\setminus I_n^2 )}\frac{P_n^t(x,y)}{\pi_n(y)}- \frac{P_n^t(a_n,b_n)}{\pi_n(b_n)} \ge 0.
$$
\end{itemize}
Then\begin{equation}
\label{eq: sepcutoff0}
\lim_{n\to \infty} \sup_{t\ge 0} |\dsepn(t)-\mathbb{P} [T_{Z_{n}}^{a_n,b_n}>t]|=0. 
 \end{equation}
 In particular, there is separation cutoff if and only if $T_{Z_{n}}^{a_n,b_n} $ is concentrated around its median.
\end{proposition}

\begin{remark} Note that the results presented above are generalizations of those presented in Section \ref{sec: Pre}.
Hence we shall only prove the more general versions in the Appendix. 
\end{remark}
\begin{remark} Similarly to Remark \ref{rem: distance}, condition (ii) of Proposition \ref{cor: sepcutoffcriterion} is satisfied in Examples \ref{ex: 4} and \ref{ex: 5} due to laziness and the fact that $\lim_{n \to \infty} \min \{t:P_n^t(a_n,Z_n)>0 \}=\infty$.
\end{remark}
\begin{lemma}
\label{lem: CS}
For any reversible
Markov chain, $(\Omega,P,\pi)$ and any $x,y \in \Omega$ and $s,t \ge 0$,
\begin{equation}
\label{mixingbndonsep}
P^{s+t}(x,y)/\pi(y) \ge (1-\|\Pr_x^t -\Pr_y^s  \|_{\mathrm{TV}})^{2}.
 \end{equation}
 In particular, if $d_x(t)+d_y(s)\le 1 $, then
 \begin{equation}
\label{mixingbndonsep2}
P^{s+t}(x,y)/\pi(y)  \ge (1-d_x(t)-d_y(s))^{2}.
 \end{equation}
 \end{lemma}
 \begin{proof}
Let $f(z):= \sqrt{P^{t}(x,z)P^{s}(y,z)}/\pi(z)$. By reversibility and Jensen's inequality,

\begin{equation*}
\begin{split}
& \frac{P^{s+t}(x,y)}{\pi(y)}=\sum_{z}\pi(z) \frac{P^{t}(x,z)P^{s}(z,y)}{\pi(z) \pi(y)} 
= \sum_z \pi(z) f^2(z) \ge \left( \sum_z \pi(z) f(z)  \right)^2 \\ & \ge \left( \sum_z \min ( P^{t}(x,z) , P^{s}(y,z))
\right)^2=(1-\|\Pr_x^t -\Pr_y^s  \|_{\mathrm{TV}})^{2}.
\end{split}
\end{equation*}
\eqref{mixingbndonsep2} follows from \eqref{mixingbndonsep} by the triangle
inequality.
\end{proof}

\subsection{Building blocks of our constructions}

Let us now describe the building blocks of our constructions. We assume for simplicity that $n$ is an even integer.
To produce the analog of a biased nearest-neighbor random walk, our constructions must include structures which look like regular trees (for which the SRW has a bias towards the leaves). We must also care about adding some extra connections to avoid producing dead-ends on the leaves (which could lead to a small Cheeger constant). Finally, we must introduce extra symmetries to ensure that the center of mass is balanced seen from all vertices which are sufficiently far from it. Finally, we ``stretch" the edges which are far away from the center of mass (that is, replace each such edge by a path of length $L$, for some fixed large constant $L$), to ensure that the worst-case total-variation and separation distances are obtained by vertices which are far away from the center of mass (which is balanced, seen from those vertices).

\medskip

\emph{Step 1:} Let $\T_a=(V_{a},E_{a})$ be a binary tree of depth $n$ rooted at $a$ (in the rest of the construction, we keep calling $a$ the {\em \textbf{root}}, even though the graph will no longer be a tree).
Replace each edge between a pair of vertices belonging to the first $n/2$ generations of  $\T_a$ by a path of $L$ edges, where $L$ is an integer which does not depend on $n$. As $L$ shall remain fixed we omit the dependence in $L$ from our notation. In the course of the proof 
we will have to require $L$ to be sufficiently large for the purpose of applying a certain crude $\ell_2$
estimate. 
 We call the obtained graph $H^1_n$. It is a tree rooted at $a$ and we denote its set of leaves by 
$$\cL_n:= (u^1,\ldots,u^{2^n}),$$  ($\cL_n$ stands for the $n$-th generation of $\cT_a$), where the labels are chosen in an arbitrary fashion.

\medskip

On $H^1_n$ the walker starting from $a$ will have a bias towards the set of leaves, which can be considered as the center of mass of these graph, since it contains a positive proportion of the vertices. 
The parameter $L$ here is present only to make the walk slower (the  expected number of steps to cross an $L$-path is $2L^2$, i.e. 
 if $v\in H^1_N$ is either the root $a$ or a vertex of degree $3$ adjacent to three degree $2$ vertices
$\mathbb{E}_v[\inf \{t\ : \ D(X_t,v)=L \}]=2L^2 $ where $D$ denotes the graph distance). 
This shall assist us in verifying that the worst-case total-variation and separation distances are obtained by vertices which are far away from the center of mass. 

The problem of this construction is that seen from a vertex which is not $a$ the set of leaves is not balanced. To cope with this defect, we add $n$ extra "generations" of vertices, which make the center of mass balanced from "many" starting points.

\medskip

\emph{Step 2:} 
For all $1 \le m \le n$ we label the vertices of the "$n+m$-th generation" (they are at distance $(L+1)n/2+m$ from $a$) as follows 
 $$\cL_{n+m}:=\{u_{i_1,\ldots,i_m}^k:i_1,\ldots,i_m \in [4],k \in [2^{n-m}] \}$$ 
 and we connect them to generation $n+m-1$ using the following scheme: for all $k \in [2^{n-m}]$
  $u_{i_1,\ldots,i_{m-1},1}^k,u_{i_1,\ldots,i_{m-1},2}^k,u_{i_1,\ldots,i_{m-1},3}^k,u_{i_1,\ldots,i_{m-1},4}^k $
 are connected to $u_{i_1,\ldots,i_{m-1}}^{2k-1} $ and  $u_{i_1,\ldots,i_{m-1}}^{2k}$. We call the obtained graph $H_1^n$.
 The "center of mass" of $H^2_n$ is the set $\cL_{2n}$ (it bears roughly half of the total mass of $H_n^2$), which is balanced seen from 
 any vertex in $H^1_n$.
 
 \medskip
 
 \emph{ Step 3.1 and 3.2:}
We now want to plug (attach) to the leaf set of $H^1_n$  ``two paths'' with different speeds (to have something similar to the structures present in 
Examples \ref{ex: 2} and \ref{ex: 3}).
The construction is the following (see Figure \ref{fig:tree}): 
\begin{itemize}
 \item [(i)] We start with a rooted binary tree $T$ of depth $n$ (assume $n\ge 4$). And let us call $1$ and $2$ the two neighbors of the root and $T_1$ and $T_2$ 
 the subtrees rooted at $1$ and $2$, respectively.
 \item [(ii)] In $T_1$ we add edges between any pair of vertices which have a common ancestor and are not leaves.
 \item [(iii)] Finally we assign labels to the leaf sets of $T_1$ and $T_2$ in a way that the two labeled trees (prior to step (ii) that is) are isomorphic  (see e.g. Figure \ref{fig:tree})
 and we merge each leaf of $T_1$ to the leaf of $T_2$ with the same label.
 We let $\cT_n$ denote the obtained graph.
 \item [(iv)] We let $\cT'_n$ denote the graph which is obtained by the same construction, in which we also add edges within $T_2$ in step $(ii)$ using the same role as for $T_1$ (see Figure \ref{fig:tree}).
\end{itemize}
To each vertex $v\in \cL_{2n}$, we glue a copy of $\cT_n$  ($v$ is merged with the root of $\cT_{n}$ and we obtain $H^{3,1}_n$).
If we glue a copy of $\cT'_n$ (to each $v\in \cL_{2n}$,) instead of $\cT_n$, we obtain $H^{3,2}_n$.
For both graphs we call $\cL_{3n}$ the set of vertices at distance $(L+5)n/2$ (i.e. maximal distance) from $a$.

\begin{figure}[h]
\begin{center}
\leavevmode
\epsfxsize =15 cm
\psfragscanon
\psfrag{Root}{\tiny Root}
\epsfbox{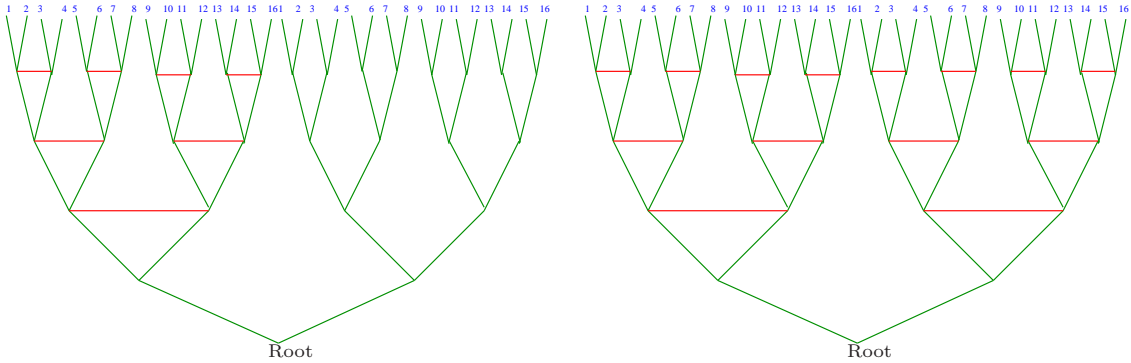}
\end{center}
\caption{\label{fig:tree} 
Representations $\cT_n$ (on the left) and $\cT'_n$ (on the right) for $n=4$. The red edges are those added in step $(ii)$. 
On step $(iv)$ leafs with the same label are merged.
}
\end{figure}

\medskip

Finally we want to link together all the vertices of $\cL_{3n}$ in order to avoid dead-ends in the graph.
We choose to link them together using an explicit expander (see e.g.~\cite{cf:Ajtai, cf:Reingold} for examples of explicit construction of expanders) 
so that (total-variation) mixing occurs rapidly once $\cL_{3n}$ is reached.

\emph{ Step 4:} 
We let $F_n=(V_n,E_n)$ be a family of explicit $3$-regular $c$-lazy expanders  with $V_n= [2^{3n-1}]$.
We glue together $G_n$ and $H^{3,i}_n$ ($i=1,2$) without adding vertices by identifying $V_n$ with $\cL_{3n-1}$.
More precisely, we start with a copy of $H_{n}^{3,i}$ with root $a$. We label the vertices of $\cL_{3n}$  by $z_1,\dots,z_{2^{3n-1}}$ (the labeling is arbitrary).   
We then connect $z_i$ with $z_j$ if and only if $\{i,j\}\in E_n$.
We call the final result of our construction $H^{4,i}_n$ ($i=1,2$). We call $a$ the root of $H_n^{4,i}$ ($i=1,2$).

\medskip

With some efforts and using the tools developed in the following sections,
the reader can check that the lazy SRW on $H^{4,1}_n$ exhibits pre-cutoff but not cutoff in total-variation. 
This is a SRW version of Aldous' counter-example.

\subsection{A sequence of Lazy SRW on bounded degree expanders with total-variation cutoff and no separation cutoff}
The following is a modification of Example \ref{ex: 2} into a sequence of lazy SRWs on a sequence of bounded degree graphs.
\begin{example}
\label{ex: 4}
Take a copy of $H_n^{3,1}$ with root $b$ and a copy of $H_n^{3,2}$ with root $a$. We glue together the two by
merging the vertices of $\cL_{3n}$ (of both graphs): we give labels $z_1,\dots,z_{2^{3n-1}}$ to the vertices lying in $\cL_{3n}$ of each of the two graphs, 
and then merge each pair of vertices who share the same label. Finally, we build extra-connections between $z_1,\dots,z_{2^{3n-1}}$
using an expander graph $F_n$ with $2^{3n-1}$ vertices, like in Step 4. We let $G^1_n:=(V_{n}^{1},E_n^1)$ denote the obtained graph.

\end{example}

\begin{figure}[h]
\begin{center}
\leavevmode
\epsfxsize =14 cm
\psfragscanon
\psfrag{Stretched}{\tiny Strectched edges}
\psfrag{hatG}{\tiny $\hat G^1_n$}
\psfrag{a}{\small $a$ }
\psfrag{b}{\small $b$ }
\psfrag{H0=An}{\tiny $H^1=A$}
\psfrag{H0=Bn}{\tiny $H^1=B$}
\psfrag{H1}{\tiny $H^2$}
\psfrag{H32}{\tiny $H^{3,2}$}
\psfrag{H31}{\tiny $H^{3,1}$}
\psfrag{Z}{\tiny $Z$}
\epsfbox{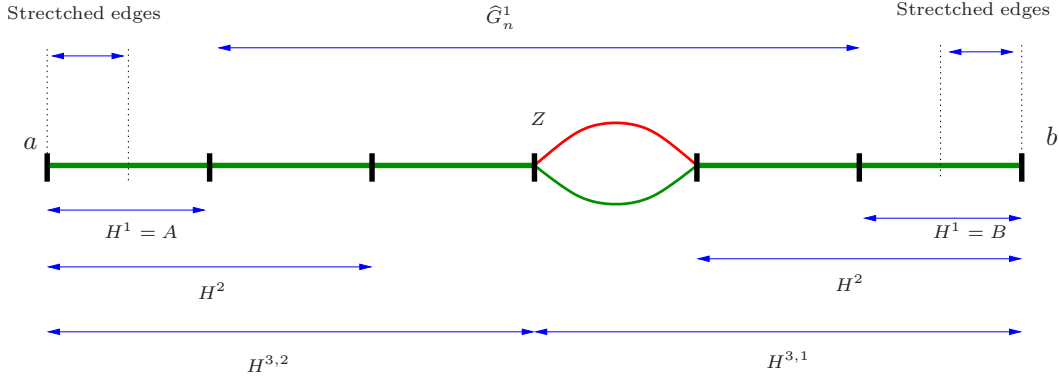}
\end{center}
\caption{\label{schematic2} 
Schematic representation of Example \ref{ex: 4}. In the construction of $H^{3,1}_n$, the asymmetry of $\cT_n$ produces two different paths to reach
the center of mass $Z_n$, with different speed. This produce an absence of concentration for the hitting time of $Z$ starting from $b$.
}
\end{figure}

In order to apply Propositions \ref{prop: hitcutoff} and \ref{cor: sepcutoffcriterion}, we need to identify which vertices and sets will play which role.
\begin{itemize}
\item The center of mass $Z_n$ is given by the $2^{3n-1}$ vertices which are linked by the expander.
\item  $a$ is the vertex which
maximizes (stochastically) the hitting time of $Z_n$.

\item The pair of vertices $(x,y)$ which (up to negligible terms) attains the minimum for $P_{n}^t(x,y)/\pi_{n}(y) $ (for all $t \ge 0$) is given by $(a,b)$.
\item The sets $A_n$ and $B_n$ are chosen to be the largest set of points around $a$ and $b$ (resp.) such that $Z_n$ is balanced seen from  $I_n:=A_n\cup B_n$.
Namely, these are the vertices within respective distance $(L+1)n/2$ from $a$ and $b$ (the vertices of $H^0_n$ in both $H_n^{3,1}$ and $H_n^{3,2}$). 
Indeed, due to step $2$ of the construction, the set $\cL_{2n}$ of $H_n^{3,1}$, respectively,   
$H_n^{3,2}$ (i.e.~the collection of vertices whose distance from $a$ (resp.~$b$) is  $(L+3)n/2$) is balanced seen from $A_n$, resp.~$B_n$. 
This implies that the distribution of $X_{T_{Z_n}}$ is uniform on $Z_n$. Step $(iv)$ of the construction of $\cT_n$ is there to guaranty that 
$T_{Z_n}$ and $X_{T_{Z_n}}$ are independent (and hence that $Z_n$ is balanced seen from $A_n$ and $B_n$).
\end{itemize}
It is then not difficult to check (cf.~Figure \ref{schematic2}) from the construction that assumptions $(i)-(iii)$ resp. $(i)-(v)$ of Propositions 
\ref{prop: hitcutoff} and \ref{cor: sepcutoffcriterion}, are satisfied. 

\medskip

Moreover, the hitting time of $Z_n$ from $a$ is concentrated around $(17+3L^2)n$, while from $b$
it satisfies that 

\begin{equation}\label{hittingZ}
\lim_{n\to \infty} \Pr_b\left [T_{Z_n}\ge s n \right]=\begin{cases} 1 \text{ if } s <15+3L^2,\\
1/2 \text{ if } s\in (15+3L^2, 17+3 L^2),\\
0 \text{  if } s>17+3L^2.
\end{cases}
\end{equation}
We want to prove that the system displays cutoff in total-variation around time $(17+3L^2)n$, and that the asymptotic behavior for the separation distance is given by
\begin{equation}
 \lim_{n\to \infty} \dsepn(sn)=\lim_{n\to \infty} \Pr_b\left [T_{Z_n}\ge (s-17+3L^{2}) n \right]=\begin{cases} 1 \text{ if } s< 32+6L^2,\\
                        1/2 \text{ if }       (32+6L^2,34+ 6L^2),\\
                        0  \text{ if } s> 34+ 6L^2.
                              \end{cases}
\end{equation}
The only thing we have to do to prove these statements is to verify condition $(iv)$ in Proposition \ref{prop: hitcutoff} and condition $(vi)$ of Proposition 
\ref{cor: sepcutoffcriterion} (resp.). The only delicate point is to show that for 
starting points outside of $I_n$ the walk mixes rapidly. I.e. that there exists an absolute constant $C>0$, which does not depend on $L$, such that
\begin{equation}\label{fastmix}
 \lim_{n\to \infty} \max_{v\notin I_n} d_{v}^{(n)}(\lceil C n \rceil)= 0.
\end{equation}

Before proving \eqref{fastmix} let us explain how we use it to verify the remaining conditions.
Note that if $L$ is chosen to be sufficiently large (i.e.~such that $(17+3L^2)>C$) then \eqref{fastmix} implies condition $(iv)$ of Proposition  \ref{prop: hitcutoff}.

\medskip

For condition $(vi)$ of Proposition \ref{cor: sepcutoffcriterion}, for the case $x\in \gO_n$ $y\notin I_n$, we use Lemma \ref{lem: CS}
and the total-variation cutoff result to show that for $t\ge  (18+3L^2+C)$
\begin{equation}
P_{n}^t(x,y)/\pi_n(y)   \ge  1- 2 \left( d^{(n)}_x(18+3L^2)+ d^{(n)}_y(Cn) \right) ,
\end{equation}
which is uniformly close to one. 

This yields the right condition provided $32+6L^2 > 18+3L^2+C$ (which can obviously be fulfilled by picking $L$ to be sufficiently large).
We now treat the case where both $x$ and $y$ lie in $A_n$ (whose analysis does not rely on \eqref{fastmix}). We use Lemma \ref{lem: pathsdecompositions} with $Z=Z'_n$ chosen to be the set of vertices within distance 
$(L+3)/2n$ from $a$ (corresponding to $\cL_{2n}$ in the copy of $H_n^{3,2}$). Recall that by construction this set is balanced seen from $A_n$.
By \eqref{eq: septhroughz11} we have that 
\begin{equation}
 P_{n}^t(x,y)/\pi_n(y) \ge \bbP\left[ T^{x,y}_{Z'_n} \le t\right].
\end{equation}
Moreover,  for any $\gep>0$
\begin{equation}
 \lim_{n\to \infty}\max_{x,y\in A_n}\bbP\left[ T^{x,y}_{Z'_n} \le (6L^2+18+\gep)n\right]=1
 \end{equation} 
 and this suffices to conclude that condition $(vi)$ of Proposition 
\ref{cor: sepcutoffcriterion} indeed holds.
 
 Now let us prove \eqref{fastmix}.
 We want to use a simple $\ell_2$ bound using the Poincar\'e inequality (see Lemma \ref{lem: contraction}).
 The issue is that the spectral gap of our graph is rather small (of order $L^{-2}$) due to the presence of stretched edges.
However starting outside of $I_n$ the walk has a very small chance to visit the part of the graph where the edges are stretched, 
 before the walk is already extremely mixed.
 Hence our idea is to apply the $\ell_2$ bound for the walk on a smaller graph which corresponds to the vertices which are likely to be visited.
 This graph will have no stretched edges and a spectral gap which is bounded away from zero and does not depend on $L$.
 
 \medskip
 
 We let $\hat G^1_n=(\hat V_n,\hat E_n)$ denote the graph which is obtained from $G^1_n$ 
when all the vertices within distance $Ln/2+1$ from $a$ and $b$ have been deleted, together with all edges connected to them.
First we observe that the Cheeger constant associated to $\hat G^1_n$ is large (i.e.~it is bounded from below by some positive absolute constant, 
which is independent also of $L$), see e.g.\
Lemma 2.1 in \cite{cf:LS} for a proof.
 
 \begin{proposition}
\label{prop: controllingcloseverticestoZn}
Let $\kappa:=\left(\min (c/3,1/18)\right)^2/2 $. Then 
\begin{equation}
\label{eq: kappa}
\mathrm{ch}_{\mathrm{Lazy}}(\hat G^1_n) \ge \sqrt{ 2\kappa } .
\end{equation}
Consequently, the relaxation-time of the lazy SRW on $\hat G^1_n$, $\hat \rel^{(n)}$, satisfies
\begin{equation}
\hat \rel^{(n)}\le \kappa^{-1}
\end{equation}
\end{proposition}

If we let $\hat \Pr_x^t$ and $\hat \pi_n$ refer to the distribution at time $t$ and at  equilibrium for the walk on $\hat G^1_n$,
this implies (by Lemma \ref{lem: contraction}) that for $x\in \hat V^1_n$, for all $t\ge  n \kappa^{-1}\log 9$.
\begin{equation}
\|\hat \Pr_x^t-\hat \pi_n \|_{\TV}\le \frac{1}{\min_{y} \hat \pi_n(y)}e^{-\kappa t} \le \left( \max_{v \in \hat V_n} \deg v \right)|\hat V_n| 9^{-n}   \le 
6 (8/9)^n.
\end{equation}
What remains to be proven is that if one considers $\hat V^1_n$ as a subset of $V^1_n$, then  for any $x\in V^1_n\setminus I_n$, the distances
$\|\hat \Pr_x^t-\hat \pi_n \|_{\TV}$ and $\|\Pr_x^t-\pi_n \|_{\TV}$ are very close.
Note that 
\begin{equation}
\|\Pr_x^t-\pin \|_{\TV}\le \| \Pr_x^t- \hat \Pr_x^t \|_{\TV}+\|\hat \Pr_x^t-\hat \pi_n \|_{\TV}+\|\pi_n-\hat \pi_n \|_{\TV},
\end{equation}
The term $\|\pi_n-\hat \pi_n \|_{\TV} $ is  exponentially small in $n$ because only an exponentially small fraction of the vertices of $G_n^1$ 
lie outside of $\hat G^1_n$.
Now if one lets $T_{\partial \hat V^1_n}$ denote the hitting time of 
$$\partial \hat V^1_n:= \{\, x\in  V_n^1 \setminus \hat V^1_n \, : \, \exists y \in \hat V^1_n, \ x \sim y\, \},$$
(recall that $V_n^1$ is the vertex set of $G_n^1$) we have (by a standard coupling argument) that
\begin{equation}\label{eq: hitime}
 \| \Pr_x^t- \hat \Pr_x^t \|_{\TV}\le \Pr_x[ T_{\partial \hat V^1_n}\le t]\le \sum_{i=1}^t \Pr^i_x\left( \partial \hat V^1_n \right).
\end{equation}

Now if $x\in V^1_n\setminus I_n$, it lies at distance of at least $n/2$ from $\partial \hat V^1_n$ and has to overcome a drift to reach it.
For this reason it should take time which is exponentially large in $n$.
More rigorously, we let $\gO_x$ be the set of vertices $y\in V^1_n$ such that there exists a graph automorphism of $G^1_n$ preserving $a$ and $b$ which maps $x$ to $y$
(in most cases it is just a pedantic manner to describe the set of points at a fixed distance from $a$, but we have to introduce this definition due to the lack of symmetry 
of the $b$-side).
Note that $|\gO_x|/| \partial \hat V^1_n |\ge 2^{n/2}$ if $x\notin I_n$.
Hence we  have for all $i >0$ and $x \notin I_n$ that 
\begin{equation}\label{eq: OKI}
\Pr^i_x\left( \partial \hat V^1_n \right)= \frac{ \sum_{y\in \gO_x} \pi_n(y)\Pr^i_y\left( \partial \hat V^1_n \right)}{\pi_n(\gO_x)}
\le \frac{\pi_n( \partial \hat V^1_n )}{\pi_n(\gO_x)} \le \max_{v \in V_n^1}\deg (v) \frac{| \partial \hat V^1_n |}{|\Omega_x|} \le \frac{6}{2^{n/2}}.
\end{equation}

where in the first inequality we have used the stationarity of $\pi_n$, 
$$\sum_{y\in V^1_n} \pi_n(y)\Pr^i_y\left( \partial \hat V^1_n \right)= \pi_n( \partial \hat V^1_n).$$
%
%
%
%
%
%
%

Plugging this in \eqref{eq: hitime} we obtain \eqref{fastmix} or more precisely:

\begin{corollary}\label{cor: l2mox}
Set $t_n:= \lceil n \kappa^{-1}\log 9 \rceil $.
Then
\begin{equation}
\label{eq: rapidmixingoutsidestretch}
\lim_{n\to \infty} \max_{x \in V_n^1 \setminus I_n} \|\Pr_x^{t_n}-\pin \|_{\TV}=0.
\end{equation}
\end{corollary}
%
%
%
%
%

\subsection{A sequence of Lazy SRW on bounded degree expanders with separation
cutoff and no total-variation cutoff}
The following is a modification of Example \ref{ex: 3} into a sequence of lazy SRWs on a sequence of bounded degree graphs.
\begin{example}
\label{ex: 5}
Take a copy of $H^{4,1}_n$ with root $a$ and a copy of $H_n^1$ with root $b$.
We glue them together as follows: we give labels in $[2^{2n}]$ to the vertices in $\cL_{2n}$ in the two graphs and 
merge the vertices which share the same labels. We denote the set of merged vertices by $Z_n'$ (this is the set of vertices of distance $(L+3)n/2$ from $a$ and $b$). Let $G^2_n$ denote the obtained graph.
\end{example}

\begin{figure}[h]
\begin{center}
\leavevmode
\epsfxsize =7 cm
\psfragscanon
\psfrag{a}{\tiny $a$}
\psfrag{b}{\tiny $b$}
\psfrag{H0=An}{\tiny $H^1_n=A_n$}
\psfrag{H0=Bn}{\tiny $H^1_n=B_n$}
\psfrag{H31}{\tiny $H^{3,1}_n$}
\psfrag{H1}{\tiny $H^{2}_n$}
\psfrag{Z}{\tiny $Z$}
\psfrag{Z'}{\tiny $Z'$}
\epsfbox{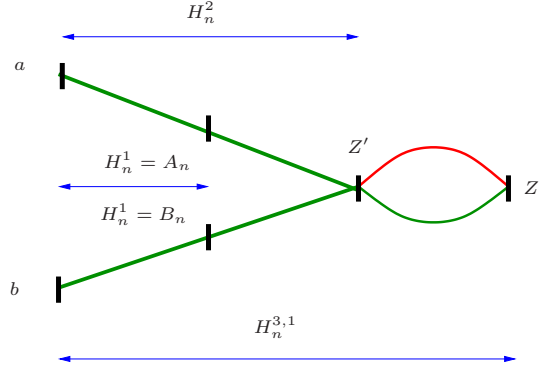}
\end{center}
\caption{\label{schematic} 
Schematic representation of Example \ref{ex: 5}. Here the asymmetry of $\cT_n$ is used to avoid cutoff in total-variation, in a similar manner to what is done in Example 
\ref{ex: 3}.
However, as in Example 
\ref{ex: 3},  the separation mixing-time is determined by the behavior of $T_{Z'}^{a,b}$ in the large deviation regime. Note that $Z'$ is a set of small equilibrium measure (it has 
$4^n$ vertices whereas
the full graph has order $8^n$ vertices).
}
\end{figure}

The reader can easily check that here $a$ and $b$ play symmetric roles. We let $A_n$ and $B_n$ denote the vertices within distance 
$(L+1)n/2$ from $a$ and $b$, respectively. Moreover,

\begin{itemize}
\item The center of mass $Z_n$ is given by the $2^{3n-1}$ vertices which are linked by the expander (which are the vertices belonging to $\cL_{3n}$ of $H^{4,1}_n$).
\item $Z_n$ is balanced seen from $A_n\cup B_n$.
\item $a$ and $b$ maximize (stochastically) the hitting time of $Z_n$.
\end{itemize}
It is then not difficult to check (see Fig.\ref{schematic}) from the construction that assumptions $(i)-(iii)$ Proposition
\ref{prop: hitcutoff} are satisfied. Assumption $(iv)$ can be showed to be satisfied as in the previous example by using an $\ell_2$ 
bound for the graph in which points within distance $Ln/2$ of $a$ and $b$ have been deleted.

\medskip

The asymptotic behavior of the hitting time of $Z_n$ from $a$ (or $b$) is once again given by 
\eqref{hittingZ} and hence the system does not display cutoff in total-variation.

\medskip

For cutoff in separation, we cannot use Proposition \ref{cor: sepcutoffcriterion}. 
We use instead Lemma \ref{lem: pathsdecompositions}, and the relevant set to hit is $Z'_n$.
This set is balanced seen from $I_{n}:=A_n\cup B_n$    and thus is the relevant one for the purpose of computing the separation mixing time.
An analog of the analysis performed for Example \ref{ex: 3}, does the job.  To control the quantity 
$P_{n}^t(x,y)/\pi_n(y)$  when one of $x$ and $y$ (or both) does not belong to $A_n\cup B_n$ we use an 
$\ell_2$ estimate (in conjunction with Lemma \ref{lem: CS}) for  the subgraph $\hat G^2_n$ obtained by deleting the stretched edges in $G^2_n$, similarly to what we have done in the analysis of Example \ref{ex: 4}.

\subsection{Proof of Remark \ref{prop: window}}
\label{sec: Rem 1.6}
Part (i) follows from the analysis of Example \ref{ex: 4}. 
We  shall prove now that part (ii) is satisfied by Example \ref{ex: 5}.

We denote by $\pi_{Z'}$ the distribution of $\pi_n$ conditioned on $Z'$ (suppressing the dependence on $n$). By \eqref{eq: septhroughz21} we have that for all $t$ and every $x \in A_n $ and $y \in B_n$
\begin{equation}
\label{eq: lowerbnd1}
P_{n}^t(x,y)/\pi_{n}(y)
=\sum_{k \le t }\mathbb{P}[T_{Z'}^{x,y}=k]\Pr^{t-k}_{\pi_{Z'}}(Z')/\pi_{n}(Z') \ge \mathbb{P}[T_{Z'}^{x,y}=t]/\pi_{n}(Z').
\end{equation}

We know from the previous analysis of Example \ref{ex: 5} 
that for the separation distance to equilibrium only $(x,y) \in A_n \times B_n$ matter, or more precisely
\begin{equation}
\lim_{n\to \infty} \sup_{t \ge 0} |\dsepn(t)- \max(0, 1- \min_{(x,y) \in A_n \times B_n} P_{n}^t(x,y)/\pi_{n}(y))|=0.
\end{equation}
Hence setting $$t^n_{\eta}(x,y):=\min \{t \, : \, P_{n}^t(x,y)/\pi_{n}(y)
 \ge 1-\eta \}$$
we prove that cutoff window is constant by proving that, for all $\gep>0$, there exist some $n_{\gep} \in \N $ and some absolute constant $C_2$ such that for all $n\ge n_{\gep}$ and all $ (x,y) \in A_n \times B_n$  
\begin{equation}\label{cutoffxy}
 t^n_{\gep}(x,y)-t^n_{1-\gep}(x,y)\le C_2 |\log  \gep|.   
\end{equation}
\begin{equation}\label{endoftimes}
  \forall t\ge t_{\gep}(x,y), \quad  P_{n}^t(x,y)/\pi_{n}(y) \ge 1-\gep.  
\end{equation}
In what follows for simplicity we drop the dependence in $n$ in the notation $t_\eta (x,y)$.
Although this is not used in the analysis below (and hence not proven), we can identify $t_{1/4}(x,y)$ for all $(x,y) \in A_n \times B_n$ as follows: 
$$\max \{ |t_{1/4}(x,y)-t'(x,y)|, |t_{1/4}(x,y)-\bar t(x,y)| \} \le C_3 ,$$ where $t'(x,y):=\inf \{t: \mathbb{P}[T_{Z'}^{x,y} \le t ] \ge \pi_{n}(Z') \} $ and $ \bar t(x,y):=\inf \{t: \mathbb{P}[T_{Z'}^{x,y}=t ] \ge \pi_{n}(Z') \} $. This follows from the analysis below, together with \eqref{eq: lowerbnd1} and the exponential decay of $\Pr^{t}_{\pi_{Z'}}(Z'))-\pi_n(Z')$ as a function of $t$.

\medskip

We start by presenting some general machinery which we shall utilize in the proof of Remark \ref{prop: window}. Let $\mu$ be a distribution over $\Z$. We say that $\mu$ is {\em \textbf{Unimodal}} if for any $z_1 \le z_2 \le z_{*}$ and for any $z_1 \ge z_2 \ge z_{*}$ we have that $\mu(z_1) \le \mu(z_2)$, where $z_{*}$ is the mode of $\mu$ (i.e.~$\max_z \mu(z)=\mu(z_{*})$). We say that $\mu$ is {\em \textbf{Log-Concave}}  if $\mu^2(z) \ge \mu(z-1)\mu(z+1)$ for all $z \in \Z$ (equivalently, for all $z_1<z_2$ ($z_1,z_2 \in \Z$) we have that $\frac{\mu(z_1+1)}{\mu(z_1)} \ge \frac{\mu(z_2+1)}{\mu(z_2)}$, where $0/0$ is interpreted as $0$).
\begin{fact}
\label{fact: logconcave}
Let $\mu$ be a log-concave distribution over $\Z$. Then $\mu$ is unimodal.\end{fact}
\begin{fact}
\label{fact: logconcave2}
The family of Geometric distributions is log-concave.
\end{fact}
\begin{fact}
\label{fact: logconcave3}
The family of log-concave distributions over $\Z$ is closed under convolutions.
\end{fact}
The following representation of hitting times in birth and death chains is due to Karlin and McGregor \cite[Equation (45)]{cf:Karlin}. It was later rediscovered by Keilson \cite{cf:Keilson}. The discrete time case of this result was given by Fill \cite[Theorem 1.2]{cf:Fill}.
\begin{theorem}
\label{thm: bdhit}
Let $([n],P,\pi)$ be a lazy birth and death chain (where $[n]:=\{1,2,\ldots,n\}$). Let $P'$ be defined by $P'(i,\cdot)=P(i,\cdot)$ if $i \in [n-1]$ and $P'(n,n)=1$. Denote the non-zero eigenvalues of $I-P'$ by $0< \beta_{1} \le \cdots \le \beta_{n-1} \le 1$. Let $\xi_1,\ldots,\xi_{n-1}$ be independent random variables such that $\xi_i \sim \mathrm{Geom}(\beta_i) $ for all $i \in [n-1]$. Then the distribution of $T_n$ under $\Pr_1$ is the same as the distribution of $\sum_{i \in [n-1]}\xi_i $.
\end{theorem}    
We are now ready to prove  \eqref{cutoffxy} and \eqref{endoftimes}. For clarity of exposition, we first expose  our analysis for the special case $x=a$, $y=b$.
Consider the sequence of graphs $G^2_n$ from Example \ref{ex: 5}. Let $G_n^3$  the subgraph of $G^2_n$ whose set of vertices is given by
$$V_n^3:=\{v : \mathrm{dist}(v,\{a,b\}) \le   (L+3)n/2 \},$$
 and whose edges are those of $E^2_n$ for which both ends are in $G^3_n$ 
(Note that this graph is connected and includes $Z'$ but not any point further away from $\{a,b\}$)

\medskip

Let $(Y_t)_{t \in \Z_+}$ be lazy SRW on $G_n^3$. Consider the projection $\bar Y_t:=1+\mathrm{dist}(Y_{t},\{a,b\})$.
 Our construction implies that the projection is Markovian and thus
$(\bar Y_t)_{t \in \Z_+}$ is a lazy birth and death chain on $[1+(L+3)n/2]$.
Consequently, by Theorem \ref{thm: bdhit} and Facts \ref{fact: logconcave2}-\ref{fact: logconcave3}, 
 the law of  $T^{a,b}_{Z'}$, which is a sum of independent hitting time 
and thus of geometric variables, is log-concave.
For any $v\in V^3_n$ the distribution of $T_{Z'}$,
given that $Y_0=v$, is the same as that of $T_{1+(L+3)n/2} $ (for the chain $(\bar Y_t)$), given that $\bar Y_0=\mathrm{1+dist}(v,\{a,b\})$.     
Consequently, by Theorem \ref{thm: bdhit} and Facts \ref{fact: logconcave2}-\ref{fact: logconcave3}, the law of $T_{Z'}^{a,b}$ is log-concave.
Let $z_{*}$ be the mode of $T_{Z'}^{a,b}$.  A standard computation is sufficient to show that
\begin{equation}
\label{eq: lowerbnd1'}
|z_{*}-\mathbb{E}[T_{Z'}^{a,b}]| \le C_4 \sqrt{\mathrm{Var}(T_{Z'}^{a,b})} \le C_5 \sqrt{n}, \end{equation}
(in fact, the first inequality follows from unimodality).

Fix some $\delta>0$ sufficiently small such that $\mathbb{P}[T_{Z'}^{a,b} \le  z_*-\delta n]\gg  2^{-n}$ 
 ($2^{-n}$ is the order of magnitude of $\pi_n(Z')$). 
By a large-deviation estimate and log-concavity there is some $\alpha>1$ such that for all sufficiently large $n$ we have that $$\alpha^{\delta n} \le \frac{\mathbb{P}[T_{Z'}^{a,b} =   z_*]}{\mathbb{P}[T_{Z'}^{a,b} =   z_*-\lfloor \delta n \rfloor]} \le \left( \frac{\mathbb{P}[T_{Z'}^{a,b} =   z_*-\lfloor \delta n \rfloor +1]}{\mathbb{P}[T_{Z'}^{a,b} =   z_*-\lfloor \delta n \rfloor]} \right)^{\lfloor \delta n \rfloor} $$ hence, again by log-concavity,
\begin{equation}\label{lalapha}
\forall t\le z_*- \delta n,\quad     \frac{\mathbb{P}[T_{Z'}^{a,b} =  t +1]}{\mathbb{P}[T_{Z'}^{a,b} = t ]}>\alpha.
\end{equation}
Consequently, by \eqref{eq: lowerbnd1}
\begin{equation}
\label{eq: monotonicity3}
\forall t \le z_*-\delta n, \quad  \frac{P_{n}^{t+1}(a,b)}{\pi_{n}(b)} 
\ge \alpha \sum_{k \le t+1 }\mathbb{P}[T_{Z'}^{x,y}=k-1]\frac{\Pr^{t+1-k}_{\pi_{Z'}}(Z')}{\pi_{n}(Z')}   =  \alpha \frac{P_{n}^{t}(a,b)}{\pi_{n}(b)} 
 .
\end{equation}
As $T_{Z'}^{a,b}$ is log-concave and hence by Fact \ref{fact: logconcave} also unimodal, \eqref{eq: lowerbnd1} also yields that
\begin{equation}
\label{eq: monotonicity1}
\forall t \in [z_*-\delta n,z_*), \quad \frac{P_{n}^t(a,b)}{\pi_{n}(b)}
 \le \sum_{k \le t }\mathbb{P}[T_{Z'}^{x,y}=k+1]\frac{\Pr^{t-k}_{\pi_{Z'}}(Z')}{\pi_{n}(Z')} \le \frac{P_{n}^{t+1}(a,b)}{\pi_{n}(b)}
,
\end{equation}
and that there exist some absolute constants $c,C_{6}>0$,  $\gb\in(1,2)$ such that
\begin{equation}
\label{eq: monotonicity2}
\forall t \in [z_*, z_*+n^{2/3}), \quad \frac{P_{n}^t(a,b)}{\pi_{n}(b)}
\ge  \frac{\mathbb{P}[T_{Z'}^{a,b}=z_*+ \lceil n^{2/3} \rceil ]}{\pi_{n}(Z')} \ge c \beta^n .
\end{equation}
\begin{equation}
\label{eq: monotonicity4}
\forall t \ge z_*+n^{2/3}, \quad 1- \frac{P_{n}^t(a,b)}{\pi_{n}(b)}
\le\mathbb{P}[T_{Z'}^{a,b}> t] \le C_6/n^{1/3}.
\end{equation}
This concludes the proof of the case $(x,y)=(a,b)$  as \eqref{eq: monotonicity3} implies \eqref{cutoffxy} with $C_2:=(\log \alpha)^{-1}$ 
and \eqref{endoftimes} can be deduced from the four other equations. For general $(x,y) \in A_n \times B_n $ 
we decompose $T_{Z'}^{x,y}$ into a convolution of a log-concave distribution and some other negligible term. Let $(X_t^x)_t $ and $(X_t^y)_t $ be independent realizations of the random walk, started from respective initial vertex $x$ and $y$, defined on the same probability space. Let $T_{Z'}^{x}:=\inf \{t: X_t^x \in Z' \}$ and $T_{Z'}^{y}:=\inf \{t: X_t^y \in Z' \}$. We define $T'_x$  (and $T'_y$ in an analogous manner, using $(X_t^y)$ and $T_{Z'}^y$) as follows (with the convention $\sup \emptyset=0$)
 $$T'_x:= \sup \{ t \ : \ t<T_{Z'}^x, \mathrm{dist}(X_{t-1}^x,Z')=\mathrm{dist}(x,Z')+1 \}$$ 
Note that $T_x'$, $T_{Z'}^{x}-T_x'$, $T_y'$ and $T_{Z'}^{y}-T_y'$ are independent. We denote $T_1=T_{1}(x,y):=(T_{Z'}^x-T'_x)+(T_{Z'}^y-T_y') $ and  $T_2=T_2(x,y):=T_{x}'+T_y'$.  
By Theorem \ref{thm: bdhit} and Facts \ref{fact: logconcave2}-\ref{fact: logconcave3} the laws of $T_{Z'}^x-T_x'$ and $T_{Z'}^y-T_y'$ are log-concave (by a similar  argument to the one used before using a projection to a birth and death chain), and so $T_1$ is also log-concave (by Fact \ref{fact: logconcave3}). Observe that 
 $T_1+T_{2}$ has the same law as $T_{Z'}^{x,y}$.

Denote the mode of $T_1$ by $z_*=z_{*}(x,y)$.
Fix some $\delta>0$ sufficiently small such that $\min_{(x,y) \in A_n \times B_n} \mathbb{P}[T_{1}(x,y)\le  z_*(x,y)-\delta n]\gg 2^{-n}$. 
 Imitating the proof of \eqref{eq: monotonicity3}, using a large-deviation estimate on $ \frac{\mathbb{P}[T_{1}(x,y) =   z_*(x,y)]}{\mathbb{P}[T_{1}(x,y) =   z_*(x,y)-\lfloor \delta n \rfloor]}$ which is uniform in $(x,y)$ (the existence of such a uniform large-deviation estimate follows from the analysis of Example \ref{ex: 3}, or alternatively, by \cite[Lemma 6.2]{cf:Basu}), together with log-concavity, we get  that if $\alpha>1$ is chosen sufficiently small, then \eqref{lalapha} remains valid simultaneously for all choices of $x,y$,  if one replaces $T_{Z'}^{a,b}$ by $T_1(x,y)$ (and $z_*$ with $z_*(x,y)$).
 We argue that \eqref{eq: lowerbnd1'}-\eqref{eq: monotonicity4} can be extended (excluding the middle terms) to all $(x,y) \in A_n \times B_n $ (in the role of $(a,b)$), with the same choice of constants   for all $(x,y) \in A_n \times B_n$. 
 To extend \eqref{eq: monotonicity3} and \eqref{eq: monotonicity1}, note that 
 after conditioning on $T_2$ we can imitate the above proofs and so the extensions are obtained by averaging over $T_2$. 
 For \eqref{eq: monotonicity2}, note that by unimodality 
 $$\mathbb{P}[T_{Z'}^{x,y}=z_*(x,y)+ \lceil n^{2/3} \rceil ]/\pi_{n}(Z') \ge c_12^n 
 \mathbb{P}[T_2(x,y) \le \lceil n^{2/3} \rceil] \mathbb{P}[T_1(x,y) =z_*(x,y)+ \lceil n^{2/3} \rceil] .$$
It is not hard to show that there exists some $\gamma<2$ and $c_{2},C_{6}>0$ such that 
$$ \mathbb{P}[T_1(x,y) =z_*(x,y)+ \lceil n^{2/3} \rceil] \ge c_{2} \gamma^{-n} \text{ and }\mathbb{P}[T_2(x,y) \le \lceil n^{2/3} \rceil] \ge 1- C_{6}n^{-2/3}. $$ 
for all $(x,y) \in A_n \times B_n$ (by Markov inequality and the fact that 
$\max_{(x,y) \in A_n \times B_n} \mathbb{E}[T_2(x,y)]=O(1) $). For \eqref{eq: lowerbnd1'} 
use unimodality (first inequality) to show that for all $(x,y) \in A_n \times B_n$  
$$|z_{*}(x,y)-\mathbb{E}[T_{1}(x,y)]| 
\le C_4 \sqrt{\mathrm{Var}(T_{1}(x,y))} \le C_4 \sqrt{\mathrm{Var}(T_{Z'}^{a,b})} \le C_5 \sqrt{n}.$$ 

Lastly, for \eqref{eq: monotonicity4} 
use \eqref{eq: lowerbnd1'} and Chebyshev's inequality (by noting that $|z_{*}(x,y)-\mathbb{E}[T_{Z'}^{x,y}]| \le |z_{*}(x,y)-\mathbb{E}[T_{1}(x,y)]|
+\mathbb{E}[T_{2}(x,y)] \le C_7 \sqrt{n} $). We leave the details to the reader. \qed
\appendix

\section{Proof of Technical results}

\subsection{Basic ingredients}
In our analysis we use various times raw $\ell_2$ bounds in order to get estimates on total-variation distance.
We cite here a standard result (see e.g.~\cite[Lemma 12.16]{cf:LPW}).
\begin{lemma}
\label{lem: contraction}
Let $(\Omega,P,\pi)$ be a finite lazy  irreducible reversible Markov chain.
Let
$\mu $ be a distribution on $\Omega$ and let $\lambda_2$ be the second largest
eigenvalue of $P$. Then
\begin{equation}
\label{eq: L2contraction}
 2\|\Pr_\mu^t-\pi \|_{\TV} \le   \|\Pr_\mu^t-\pi \|_{2,\pi} \le  \lambda_2^t
\|\mu-\pi \|_{2,\pi}, \text{ for all }t \ge 0.
\end{equation}
\end{lemma}

\begin{lemma}[Hitting time from stationary tail estimates]
Let $(\Omega,P,\pi)$ be a finite   irreducible reversible Markov chain. Let $A \subset \Omega $. Then for any $t \ge0$ we have that
\begin{equation}
\label{eq: hittingfrompi}
\Pr_{\pi}[T_A > t] \le (1-\pi(A))\exp \left(-t \pi(A)/\rel \right).
\end{equation}
\end{lemma}
For a proof see \cite[Lemma 3.5]{cf:Basu} (or Proposition 3.21 in conjunction with Theorem 3.33 and Corollary 3.34 in \cite{cf:Aldous}).

\subsection{Proof of \eqref{groto2}} \label{sec: groto2}

We can assume that the chain displays pre-cutoff in separation as if not, there is nothing to prove.
We know from \eqref{eq:TVsepasymeqiv}  that for every $\gep>0$, $ \mixn(\gep) \le \sepn(\gep)  .$
Hence, in our case, it is sufficient to prove that
\begin{equation}\label{eq:thingtoprove}
 \limsup_{\gep\to 0} \limsup_{n\to \infty} \frac{\sepn(1- \gep)}{\mixn(1-2\sqrt{\gep})}\le 2,
\end{equation}
as pre-cutoff implies that $\frac{\sepn(1- 2\sqrt{\gep})}{\sepn(1- \gep)}$ tends to $1$ when $n$ 
goes to infinity and $\gep$ goes to $0$ in this order.

\medskip

We shall show that for all $n$
\begin{equation}\label{eq: mixcompa}
 \sepn(1- \gep) \le 2\mixn(1-2\sqrt{\gep})+2 \reln \log \gep^{-1}.
\end{equation}
This is sufficient to conclude as Fact \ref{fact: cutoffandtrel} ascertains that the second term is small in comparison to the first.

\medskip

To prove \eqref{eq: mixcompa}, let us introduce the following alternative way of measuring the distance to equilibrium,
\begin{equation}\label{eq:defbard}
 \bar d(t):= \max_{x,y \in \gO} \|\Pr_x(X_t\in \cdot)- \Pr_y(X_t\in \cdot)\|_{\TV}.
\end{equation}
From \cite[Lemma 19.3]{cf:LPW} (or \eqref{mixingbndonsep}) we know that for every $t$ we have 
\begin{equation}
 d_{\sep}(t)\le 1-(1-\bar d(t/2))^2.
\end{equation}
Hence if one defines $\bar{t}_{\mathrm{mix}}(\gep)$ to be the first time at which $\bar d(t)\le \gep$, we have for every $\gep>0$.
\begin{equation}
 \sep(1- \gep)\le 2 \bar{t}_{\mathrm{mix}}(1-\sqrt{\gep}).
\end{equation}
Now to conclude the proof of \eqref{eq:thingtoprove} we need to show that (under reversibility)
\begin{equation}
 \bar{t}_{\mathrm{mix}}(1-\sqrt{\gep})\le \mix(1-2\sqrt{\gep})+ \rel \log \gep^{-1}.
\end{equation}
Let us set 
\begin{equation}
t=\mix(1-2\sqrt{\gep}) \quad \text{ and } s=\rel \log \gep^{-1}.
\end{equation}
For $x\in \gO$, we set  $\mu^x:= \Pr_{x}(X_t \in \cdot )$ and 
\begin{equation}
 \mu^x_1(z):= \frac{\min ( \mu^x(z), \pi(z))}{1-\|\mu^x-\pi\|_{\TV}} \text{ and } \mu_2^x(z):=\frac{(\mu^x(z)-\pi(z))\ind_{\{\mu^x(z)> \pi(z)\}}}{\|\mu^x-\pi\|_{\TV}}.
\end{equation}
Both are probability measures and $\mu^x$ can be written as a linear combination of the two
$$\mu^x= (1-\|\mu^x-\pi\|_{\TV})\mu^x_1+ \|\mu^x-\pi\|_{\TV}\mu^x_2.$$
For $\mu$, a distribution on $\gO$,   $f,g \in \R^{\gO}$ and $1 \le p<\infty$, we introduce the notation
$$\mu P^s(x):= \sum_{y\in \gO} \mu(y) P^s(y,x),\quad P^sf(x)=\sum_{y \in \gO}P^s(x,y)f(y),\quad \|g\|_{p}=(\sum_{y} \pi(y)|g(y)|^p)^{1/p}.$$  
Assuming that $x$ and $y$ maximize the l.h.s.\ in \eqref{eq:defbard} at time $t+s$ and that 
$$\|\mu^x-\pi\|_{\TV}\ge \|\mu^y-\pi\|_{\TV}$$ then by
the triangular inequality we have that 
\begin{multline}
 \bar d(t+s)=  \|  \mu^x P^s -  \mu^y P^s \|_{\TV}\\
 \le \left(1-\|\mu^x-\pi\|_{\TV}\right) \| \mu^x_1 P^s -\mu^y_1  P^s \|_{\TV}+\|\mu^x-\pi\|_{\TV} \| \mu^x_2 P^s- \tilde \mu^y_2 P^s\|_{\TV}, 
\end{multline}
where $\tilde \mu^y_2$ is an adequate linear combination of $\mu^y_1$ and $\mu^y_2$.
Using the definition of $t$ we have $\|\mu^x-\pi\|_{\TV}\le 1-\sqrt{\gep}$, and hence we have  
\begin{equation}
  \bar d(t+s)\le 1-2\sqrt{\gep}+  \| \mu^x_1 P^s -\mu^y_1  P^s \|_{\TV},
\end{equation}
Now to estimate the second term we set $f(z)= \frac{\mu^x_1(z)}{\pi(z)}- \frac{\mu^y_1(z)}{\pi(z)}$. Observe that by reversibility $ \mu^x_1 P^s -\mu^y_1  P^s=P^sf$.
Note that by definition $|f(z)|\le (1-\|\mu^x-\pi\|_{\TV})^{-1}\le 2\gep^{-1/2}$.  Let $1=\gl_1 \ge \gl_2 \ge \ldots \ge \lambda_{|\gO|}$ be the eigenvalues of $P$ and set $\gl:=\max(\gl_2,|\gl_{|\gO|}|) $. Using the spectral-decomposition of $f$ along with the fact that $\sum_z \pi(z)f(z)=0 $ (and finally, the choice of $s$) it is standard to show that $\|P^sf \|_2 \le \gl^s \|f\|_2 \le \gl^s (2\gep^{-1/2} )\le 2 \sqrt{ \eps} $. Hence 
\begin{equation*}
 \| \mu^x_1 P^s -\mu^y_1  P^s \|_{\TV}= \frac{1}{2} \|P^sf\|_1 \le \frac12 \|P^sf\|_2 \le \sqrt{\eps}, \\ 
\end{equation*}
as desired.
\qed

\subsection{Proof of Lemma \ref{lem: pathsdecompositions}}

By decomposing over the possible values of $T_Z$, using the assumption
that $Z$ is balanced seen from $x$ and reversibility (which implies that $\Pr_{\pi_{Z}}^{s}(y)/\pi(y)
=\Pr_{y}^{s}(Z)/\pi(Z)$, for all $s$), we get that
\begin{equation}
\label{eq: TAxy1}
\begin{split}
& \frac{P^t(x,y)}{\pi(y)}=\sum_{k_1 \le t }\Pr_x[T_{Z}=k_1] \frac{\Pr_{\pi_{Z}}^{t-k_1}(y)}{\pi(y)}+
\frac{\Pr_x[ X_t=y  \text{ and } T_{Z}>t]}{\pi(y)}\\ &
\ge  \sum_{k_1 \le t }\Pr_x[T_{Z}=k_1] \frac{\Pr_{\pi_{Z}}^{t-k_1}(y)}{\pi(y)}=
 \sum_{k_1 \le t }\Pr_x[T_{Z}=k_1] \frac{\Pr_{y}^{t-k_1}(Z)}{\pi(Z)} \\ &
= 
 \sumtwo{0\le k_1 \le t}{0\le k_2\le k_1} \Pr_x[T_{Z}=k_2]\Pr_y[T_{Z}=k_1-k_2] \frac{\Pr_{\pi_Z}^{t-k_1}(Z)}{\pi(Z)}
= \sum_{k \le t }\mathbb{P}[T_{Z}^{x,y}=k]\frac{\Pr_{\pi_Z}^{t-k}(Z)}{\pi(Z)}.
\end{split}
\end{equation}
In particular, we have equality if $\mathrm{P}_{x}[T_Z<T_y]=1$ (i.e.\ in
 case $(ii)$). To conclude, for $(i)$ we note that for lazy irreducible reversible chains
that $\Pr_{\pi_Z}^{s}(Z) \ge \pi(Z)$, which can be easily verified using
the spectral decomposition and the non-negativity of the eigenvalues of $P$.
For $(ii)$ note that by \eqref{eq: L2contraction} $$\Pr_{\pi_Z}^{s}(Z) - \pi(Z) \le\|\Pr_{\pi_Z}^s-\pi \|_{\TV} \le  \frac{1}{2} \|\pi_Z-\pi\|_2 \lambda_2^s
\le\frac{\lambda_2^{s}}{2}\sqrt{(1-\pi(Z))/\pi(Z)}.$$
The first inequality in \eqref{eq: septhroughz21} is obtained by plugging the last estimate in the second term of \eqref{eq: septhroughz21}. For the second inequality in \eqref{eq: septhroughz21} it follows from the estimate 
\begin{equation*}
\sum_{k \le t }\mathbb{P}[T_{Z}^{x,y}=k]\lambda_2^{t-k} \le \left( \max_{k\ge 0} \mathbb{P}[T_{Z}^{x,y}=k] \right)\sum_{k \ge 0 }\lambda_2^{k}  =\max_k \mathbb{P}[T_{Z}^{x,y}=k]/(1-\lambda_2). \qed
\end{equation*}

\subsection{Proof of Proposition \ref{prop: hitcutoff} }

The result is mostly a consequence of the following result which relates the mixing time starting from $x$ to the hitting time of a set $Z$ balanced seen from $x$. 

\begin{lemma}
\label{cor: contractioncor}
Let $(\Omega,P,\pi)$ be a finite lazy  irreducible reversible Markov chain.
Let $Z \subset \Omega$ (we denote its complement by $Z^c$) and $x \in \Omega$.  Given $0<\epsilon < 1$. Set
\begin{equation*}
\begin{split}
t_{x,Z}(p):=&\min\{t:\Pr_x[T_{Z}>t] \le p \}, \\
s_{\eps}:= \left\lceil \rel \log \left[ \pi(Z^c) / \eps \right] /\pi(Z) \right\rceil  &\text{ and }
 r_{\eps}:=\left\lceil \rel \log \left[\pi(Z^{c})/ (\pi(Z) \eps^2)\right]/2 \right\rceil.
\end{split}
\end{equation*}
Let $s':=\max (t_{x,Z}(p)-s_\gep,0)$. Then we have 
\begin{equation}
\label{eq: hitmixeasydirection}
\|\Pr_x^{s'}-\pi \|_{\TV} > p-\gep.
\end{equation}
Moreover if $Z$ is balanced seen from $x$ then we also have that 
\begin{equation}
\label{eq: hitmix2}
\|\Pr_x^{t_{x,Z}(p)+r_{\gep}}-\pi \|_{\TV}  \le p+\eps.
 \end{equation}
\end{lemma}
\begin{proof}
The first result is proved by coupling the chain with initial distribution $\Pr_{x}^{k-s_{\gep}}$  with the
stationary chain 
($k\geq s_{\gep}$ to be determine soon). We have
\begin{equation}
\label{eq: thitinequality3}
\Pr_x[T_Z \ge k]\leq \|\Pr_x^{k-s_{\eps}}-\pi \|_{\TV}+\Pr_{\pi}[T_Z \ge s_{\eps}] \leq \|\Pr_x^{k-s_{\gep}}-\pi \|_{\TV}+\epsilon.
\end{equation}
where the last inequality is a consequence of \eqref{eq: hittingfrompi} and the choice of $s_{\eps}$. Setting $k= t_{x,Z}(p)$
we obtain the result (as if $s'=0$ there is nothing to prove).

We now prove \eqref{eq: hitmix2}. By the assumption that $Z$ is balanced seen from $x$,  for all $\ell\le t$
\begin{equation}
 \Pr_{x}^t=\Pr_x[T_Z>\ell]\Pr_x[ X_t\in \cdot \ \mid  \ T_Z>\ell]+ \!\!\!  \sum_{0 \le
i < \ell }  \!\!\!  \Pr_x[T_Z= \ell -i  ] \Pr_{\pi_{Z}}^{t-\ell+i}.
\end{equation}
By  the triangle inequality and the fact that the distance to $\pi$ decreases in time, we obtain
\begin{equation*}
\begin{split}
& \|\Pr_{x}^t-\pi \|_{\TV} \le \Pr_x[T_Z>\ell] +  \sum_{0 \le
i < \ell }\Pr_x[T_Z= \ell -i  ] \|\Pr_{\pi_{Z}}^{t-\ell+i}-\pi \|_{\mathrm{TV}}\\ & \le \Pr_x[T_Z>\ell]+ \|\Pr_{\pi_{Z}}^{t-\ell}-\pi \|_{\mathrm{TV}}
\end{split}
\end{equation*}
Using this inequality for $\ell:=t_{x,Z}(p)$ (and so $t-\ell=r_{\gep}$)  we only have to show that $\|\Pr_{\pi_{Z}}^{t-\ell}-\pi \|_{\mathrm{TV}}=\|\Pr_{\pi_{Z}}^{r_{\gep}}-\pi \|_{\mathrm{TV}} \le \gep$.
Combining \eqref{eq: L2contraction} with the definition of $r_\gep$, we have that
\begin{equation}
\label{eq: TVbndonmu}
\| \Pr_{\pi_{Z}}^{r_{\epsilon}}-\pi \|_{\TV} \le \lambda_2^{r_{\eps}}\sqrt{\pi(Z^c)/\pi(Z)}
 \le \eps.
\end{equation}

\end{proof}

We can now proceed to the proof of Proposition \ref{prop: hitcutoff}. With our assumptions on $\rel$ and $Z_n$,
Lemma \ref{cor: contractioncor} allows us to show that mixing time starting from $x$ 
and $t_{x,Z_{n}}(p)$ are equivalent when $Z_n$ is balanced seen from $x$ (i.e. for $x\in I_n$). Assumption $(iv)$ ensures that what occurs for other initial conditions does not matter and Assumption $(iii)$ establishes that $a$ is the worst initial condition.

\qed

\subsection{Proof of Proposition \ref{cor: sepcutoffcriterion}}

From Lemma \ref{lem: pathsdecompositions} and assumptions $(i)$, $(iii)$ and $(v)$
we know that 
$P_{n}^t(x,y)/ \pi_{n}(y)$ and $\bbP[T_{Z_n}^{x,y}\le t]$ differ only by a negligible amount, provided that $x\in A_n$ and $y\in B_n$.
Assumption $(iv)$ ensures then that
\begin{equation}
 \liminf_{n\to \infty} \inf_{t\ge 0} \min_{(x,y)\in A_n\times B_n} \frac{P_{n}^t(x,y)}{ \pi_n(y)}-\frac{P_n^t(a_n,b_n)}{ \pi_n(b_n)}=0.
\end{equation}
We are left checking the other cases.
Assumption $(vi)$ takes care of most of them, and leaves the case where $(x,y)\in B_n \times B_n $, for which Lemma \ref{lem: pathsdecompositions}
implies that $\bbP[T_{Z_n}^{x,y}\le t]$ is a lower bound for  $P_{n}^t(x,y)/\pi_{n}(y)$. Hence the conclusion follows by assumption $(v)$ again.
\qed

\subsection{A short alternative proof of Theorem \ref{thm: window}} 

We are going to show that there exists an absolute constant $c>0$ such that for any lazy chain
\begin{equation}
\mix(1/4)-\mix(3/4)\ge c\sqrt{\mix (3/4)}.
\end{equation}
Indeed set $t:= \mix(1/4)$ and $s:= \lfloor c\sqrt{t}\rfloor$.
A sample of the distribution of the lazy chain at time $t$ can be generated by running the non-lazy version of the chain for $\xi_{t}$ steps, 
where $\xi_{t} \sim \mathrm{Bin}(t,1/2)$ and is independent of the non-lazy version of the chain.
By the triangle inequality we have (first inequality) and a standard coupling argument (second inequality) 
\begin{equation*}
\forall t,s \ge 0, \quad d(t)-d(t+s)\le \max_{x\in \gO} \|\Pr_{x}^{t}-\Pr_{x}^{t+s}\|_{\mathrm{TV}} \le \| \xi_{t}-\xi_{t+s} \|_{\mathrm{TV}}.
\end{equation*}
Moreover, if $c$ is chosen well, we have for every $t\ge 0$ that
$\| \xi_{t}-\xi_{t+\lfloor c\sqrt{t}\rfloor} \|_{\mathrm{TV}}\le 1/2$.
 \qed

\end{document}